\documentclass{article}
\usepackage{amsmath,amsfonts,amsthm,amssymb,amscd,color,xcolor,mathrsfs,eufrak}
\usepackage{hyperref}

\colorlet{darkblue}{blue!50!black}

\hypersetup{
    colorlinks,%
    citecolor=darkblue,%
    filecolor=red,%
    linkcolor=darkblue,%
    urlcolor=magenta,%
    pdfnewwindow=true,%
    pdfstartview={FitH}
}

\renewcommand{\Re}{\mathop{\rm Re}\nolimits}

\newcommand{\p}{\partial}
\newcommand{\e}{\varepsilon}

\newcommand{\ri}{{\rightarrow}}

\newcommand{\R}{{\mathbb R}}
\newcommand{\Z}{{\mathbb Z}}
\newcommand{\IP}{{\mathbb P}}
\newcommand{\pP}{{\mathbb P}}

\newcommand{\E}{{\mathbb E}}

\newcommand{\N}{{\mathbb N}}
\newcommand{\la}{\lambda}

\newcommand{\ty}{\infty}

\newcommand{\om}{\omega}
\newcommand{\bg}{{\boldsymbol{\varUpsilon}}}

\newcommand{\de}{\delta}

\newcommand{\BBB}{{\boldsymbol{B}}}

 \newcommand{\HHH}{{\boldsymbol{H}}}

\newcommand{\MH}{{\boldsymbol{\mathfrak{{H}}}}}
\newcommand{\MU}{{\boldsymbol{\mathfrak{{A}}}}}
\newcommand{\uu}{{\boldsymbol{ {u}}}}

 \newcommand{\Ss}{{\boldsymbol {S}}}
\newcommand{\eeta}{{\boldsymbol{{\eta}}}}

\newcommand{\zzeta}{{\boldsymbol{{\zeta}}}}
\newcommand{\III}{{\boldsymbol{I}}}

\newcommand{\XXX}{{\boldsymbol{X}}}

\newcommand{\ssigma}{{\boldsymbol\sigma}}

\newcommand{\ppsi}{{\boldsymbol\psi}}
\newcommand{\UUU}{{\boldsymbol{U}}}

\newcommand{\aA}{{\cal A}}
\newcommand{\BB}{{\cal B}}
\newcommand{\CC}{{\cal C}}
\newcommand{\DD}{{\cal D}}

\newcommand{\FF}{{\cal F}}

\newcommand{\HH}{{\cal H}}
\newcommand{\KK}{{\cal K}}
\newcommand{\LL}{{\cal L}}
\newcommand{\MM}{{\cal M}}

\newcommand{\PP}{{\cal P}}

\newcommand{\VV}{{\cal V}}
\newcommand{\WW}{{\cal W}}

\newcommand{\lag}{\langle}
\newcommand{\rag}{\rangle}

\newcommand{\dd}{{\textup d}}

\newcommand{\PPPP}{{\mathfrak P}}
\newcommand{\BBBB}{{\mathfrak B}}
\newcommand{\bPPPP}{\boldsymbol{\mathfrak P}}

\newcommand{\VVV}{\boldsymbol{V}}

\newcommand{\BBBBB}{{\mathcal B}}
\newcommand{\SSS}{{\mathscr S}}

\newcommand{\fff}{{\boldsymbol{\mathit f}}}
\newcommand{\eell}{{\boldsymbol{\ell}}}
\newcommand{\uuu}{{\boldsymbol{\mathit u}}}
\newcommand{\nnn}{{\boldsymbol{\mathit n}}}
\newcommand{\vvv}{{\boldsymbol{\mathit v}}}
\newcommand{\www}{{\boldsymbol{\mathit w}}}

\newcommand{\bgg}{{\boldsymbol{g}}}

\newcommand{\lspan}{\mathop{\rm span}\nolimits}

\newcommand{\supp}{\mathop{\rm supp}\nolimits}
\newcommand{\diver}{\mathop{\rm div}\nolimits}
\newcommand{\Lip}{\mathop{\rm Lip}\nolimits}

\newcommand{\Osc}{\mathop{\rm Osc}\nolimits}

\theoremstyle{plain}
\newtheorem*{mta}{Theorem A}
\newtheorem*{mtb}{Theorem B}
\newtheorem{theorem}{Theorem}[section]
\newtheorem{lemma}[theorem]{Lemma}
\newtheorem{proposition}[theorem]{Proposition}

\theoremstyle{definition}
\newtheorem{definition}[theorem]{Definition}

\theoremstyle{remark}

\numberwithin{equation}{section}

\begin{document}
\author{V.~Jak$\check{\rm s}$i\'c\footnote{Department of Mathematics and Statistics,
McGill University, 805 Sherbrooke Street West, Montreal, QC, H3A 2K6
Canada; e-mail: jaksic@math.mcgill.ca}
\and V. Nersesyan\footnote{Laboratoire de Mat\'ematiques, UMR CNRS 8100, Universit\'e de Versailles-Saint-Quentinen-Yvelines, F-78035 Versailles, France;  e-mail: Vahagn.Nersesyan@math.uvsq.fr}
\and C.-A.~Pillet\footnote{Aix Marseille Universit\'e, CNRS, CPT, UMR 7332, Case 907, 13288 Marseille, France; 
Univerist\'e de Toulon, CNRS, CPT, UMR 7332, 83957 La Garde,  France; e-mail: pillet@univ-tln.fr}
\and A.~Shirikyan\footnote{Department of Mathematics, University of Cergy--Pontoise, CNRS UMR 8088, 2 avenue Adolphe Chauvin, 95302 Cergy--Pontoise, France; e-mail: Armen.Shirikyan@u-cergy.fr}}

\title{Large deviations from a stationary measure for a class of dissipative PDE's with random kicks}
\date{}
\maketitle

\begin{abstract}
We study a class of dissipative PDE's perturbed by a bounded random kick force. It is assumed that the random force is non-degenerate, so that the Markov process obtained by the restriction of solutions to integer times has a unique stationary measure. The main result of the paper is  a large deviation principle  for occupation measures of the Markov process in question. The proof is based on Kifer's large deviation criterion, a Lyapunov--Schmidt type reduction, and an abstract result on large-time asymptotic for generalised Markov semigroups. 


\end{abstract}

\tableofcontents

\setcounter{section}{-1}

\section{Introduction}
\label{s0}
This paper is devoted to the large deviations principle (LDP) for a class of dissipative PDE's perturbed by a smooth random force. The large-time asymptotics of solutions for the problem in question is well understood, and we refer the reader to the articles~\cite{FM-1995,KS-cmp2000,EMS-2001,BKL-2002} for the first results in this direction and to the book~\cite{KS-book} for further references and a detailed description of the behaviour of solutions as time goes to infinity. In particular, it is known that if the noise is sufficiently non-degenerate, then the Markov process associated with the problem has a unique stationary distribution, which attracts exponentially the law of all solutions. Moreover, the law of iterated logarithm and the central limit theorem hold for H\"older-continuous functionals calculated on trajectories and give some information about fluctuations of their time averages around the mean value. Our aim now is to investigate the probabilities of deviations of order one from the mean value.

Let us describe in more detail the main result of this paper on the example of the Navier--Stokes system. More precisely, we consider the following problem in a bounded domain\footnote{All the results of this paper remain true for periodic boundary conditions, in which case we assume in addition that the mean values of the velocity field and of the external force are zero.} $D\subset\R^2$ with a $C^2$-smooth boundary~$\p D$:
 \begin{align}
\dot u-\nu \Delta u+\lag u, \nabla \rag u+\nabla p &=\eta(t,x), 
\quad  \diver u =0, \quad u\bigl|_{\p D}=0, \label{E:1}\\ 
u(0,x)&=u_0(x),\label{E:3}
\end{align} 
where $\nu>0$ is the kinematic viscosity, $u = (u_1(t,x), u_2(t,x))$ is the velocity field of the fluid, $p = p(t, x)$ is the pressure, and~$\eta$ is a random external force. We assume that~$\eta(t,x)$ is a random kick force  of the form
\begin{align}
\eta(t,x)&=\sum_{k=1}^{+\ty} \de(t-k)\eta_k(x), \label{E:4}
 \end{align}
 where $\de$ is the Dirac measure concentrated at zero, and $\eta_k$ are independent identically distributed (i.i.d.) random variables defined on a probability space $(\Omega,\FF,\pP)$ with range in $L^2(D,\R^2)$ that satisfy 
 \begin{equation} \label{0.4}
\pP\{\|\eta_k\|_{L^2}\le b\}=1
\end{equation}
for some $b<+\ty$.
Problem~\eqref{E:1}, \eqref{E:3} is well posed in the space
\begin{equation} \label{ps}
H= \{u\in  L^2(D,\R^2): \diver u=0\mbox{ in $D$}, \langle u,\nnn\rangle=0\mbox{ on $\p D$}\},
\end{equation}
where $\nnn$ stands for the outside unit normal to~$\p D$. The restrictions of solutions for~\eqref{E:1}, \eqref{E:3} to integer times form a Markov chain in~$H$. As is well known (see Chapter~3 of the book~\cite{KS-book} and the references therein), this process is ergodic under rather general hypotheses on~$\eta_k$.  More precisely, suppose that there exists  an increasing sequence of finite-dimensional subspaces $H_N\subset H$ such that the law of the projection of~$\eta_k$ to~$H_N$ is absolutely continuous with respect to the Lebesgue measure, and its support contains the origin. 
Let~${\mathcal P}(H)$ be the set of all Borel probability measures on $H$ endowed with topology of weak convergence.
Then the Markov chain in question possesses a unique stationary measure~$\mu\in {\cal P}(H)$, which is exponentially mixing in the sense that the law of any solution of~\eqref{E:1} with a deterministic initial condition  converges to~$\mu$ exponentially fast in the Kantorovich--Wasserstein metric. We wish to investigate the probabilities of large deviations of the occupation measures from~$\mu$. More precisely, let 
\begin{equation} \label{0.5}
\zeta_k^\omega=\frac1k\sum_{j=0}^{k-1}\delta_{v_j},\quad k\ge1,
\end{equation}
be a sequence of random probability measures in ${\cal P}(H)$,  where~$\{v_j\}$ denotes a  stationary trajectory of the Markov chain. The following theorem is a simplified version of the main result of this paper (see Theorem~\ref{main}). 

\begin{mta}
Under the above hypotheses, the sequence $\{\zeta_k\}$ satisfies a LDP. More precisely, there is a lower semicontinuous mapping $I:\PP(H)\to[0,+\infty]$ which is equal to~$+\infty$ outside a compact subset such that 
\begin{equation} \label{0.6}
-\inf_{\lambda\in \dot\Gamma}I(\lambda)
\le \liminf_{k\to\infty}\frac1k\log\IP\{\zeta_k\in\Gamma\}
\le \limsup_{k\to\infty}\frac1k\log\IP\{\zeta_k\in\Gamma\}
\le -\inf_{\lambda\in \overline\Gamma}I(\lambda),
\end{equation}
where $\Gamma\subset\PP(H)$ is an arbitrary Borel subset, and $\dot\Gamma$ and $\overline\Gamma$ denote its interior and closure, respectively.
\end{mta}
For instance, if $f:H\to\R^m$ is a continuous mapping and $B\subset\R^m$ is a Borel subset, then taking $\Gamma=\{\sigma\in\PP(H):\int_Hf\dd\sigma\in B\}$ in inequality~\eqref{0.6},  we get  (see Section~\ref{s2.1} for a precise statement) 
$$
\exp(-c_-\,k)\lesssim\IP\biggl\{\frac1k\sum_{j=0}^{k-1}f(v_j)\in B\biggr\}\lesssim \exp(-c_+\,k)\quad\mbox{as $k\to\infty$}, 
$$
where $c_\pm=c_\pm(f,B)\ge0$ are some constants (not depending on~$k$) that can be expressed in terms of the rate function~$I$.

Let us mention that the LDP is well understood for finite-dimensional diffusions and for Markov processes with compact phase space, provided that the randomness is sufficiently non-degenerate and ensures mixing in the total variation norm. This type of results were first obtained by Donsker and Varadhan~\cite{DV-1975} and later extended by many others. A detailed account of the main achievements can be found in the books~\cite{FW1984,DS1989,ADOZ00}. 

In the context of randomly forced PDE's, the problem of large deviations was studied in a number of papers. Most of them, however, are devoted to studying PDE's with vanishing random perturbation and provide estimates for the probabilities of deviations from solutions of the limiting deterministic equations. We refer the reader to the papers~\cite
{freidlin-1988,sowers-1992a,sowers-1992b,chang-1996,CR-2004,CR-2005,SS-2006,CM-2010} and the references therein for various results of this type, including the asymptotics of stationary distributions when the amplitude of the perturbation goes to zero. To the best of our knowledge, the only papers devoted to large deviations from a stationary distribution in the case of stochastic PDE's are those by Gourcy~\cite{gourcy-2007a,gourcy-2007b}. Using a general result due to Wu~\cite{wu-2000}, he established the LDP for occupation measures of stochastic Burgers and Navier--Stokes equations, provided that the random force is white in time and sufficiently irregular in the space variables. The present paper gives a first result on large deviations from a stationary distribution  for PDE's with a {\it smooth\/} random perturbation. 

Let us note that, in Gourcy's papers, the set of measures is endowed with the {\it $\tau$-topology\/} which is generated by the duality with respect to bounded Borel functions (and is much stronger than the weak topology used in our paper). This enables one to apply the LDP to physically relevant observables that are not continuous on the energy space. Under our assumptions, the LDP is not likely to hold for the $\tau$-topology. However, the results established in this paper can be applied to derive the LDP for functionals that are continuous on higher Sobolev spaces. Furthermore, using the Dawson--G\"artner theorem~\cite{DG-1987}, we establish the following result on large deviations in the space of trajectories (also called process level LDP). Let us denote by~$\HHH$ the space of sequences $\uuu=(u_j,j\ge0)$ with $u_j\in H$ and endow it with the Tikhonov topology. Given a stationary trajectory $\vvv=\{v_j\}$  for the Markov chain associated with~\eqref{E:1}, we define the sequence of occupation measures 
\begin{equation}
\zzeta_k^\omega=\frac{1}{k}\sum_{j=0}^{k-1}\delta_{\vvv_j},
\end{equation}
where we set $\vvv_j=\{v_i,i\ge j\}$. 

\begin{mtb}
Let us assume that the above-mentioned hypotheses are satisfied. Then the LDP holds for~$\zzeta_k$ with a rate function $\III:\PP(\HHH)\to[0,+\infty]$.
\end{mtb}

In conclusion, let us mention that the LDP discussed above remains valid  in the case of unbounded perturbations; this question will be addressed in a subsequent publication. We also 
remark  that this paper is a first step of a research program whose aim is to develop a large deviation theory for dissipative PDE's with random perturbation and to justify the Gallavotti--Cohen fluctuation principle for some relevant functionals; cf.~\cite{GC-1995}. 

\medskip
The paper is organised as follows. In Section~\ref{s1} we introduce the model, state our results, describe applications, and outline the schemes 
of the proofs. Section~\ref{s3} deals with the large-time asymptotics of generalised Markov semigroups. 
A central technical part of the proof is the verification of  uniform Feller property for a suitable family 
of  semigroups.
This verification is based  on the Lyapunov--Schmidt reduction and is carried out in Section~\ref{sec-unif}.
The proofs of the main results are given in Sections~\ref{s4} and \ref{s6}.
 Finally, three auxiliary results used in the main text are recalled in the Appendix. 

\subsection*{Acknowledgments} 
The research of VJ was supported by NSERC.
The research of VN and AS was supported by the ANR grants EMAQS and STOSYMAP (No.~ANR 2011 BS01 015 01). The research of CAP was partly
supported by ANR grant 09-BLAN-0098.

\subsection*{Notation}
Let~$\Z$ be the set of integers, let~$\Z_l$ be the set of integers less than or equal to~$l$, and let~$X$ be a Polish space with a metric $d_X(u,v)$. We denote by~$X^k$ the direct product of~$k$ copies of~$X$, by~$\XXX=X^{\Z_0}$ the space of sequences $(u_k, k\in\Z_0)$ with $u_k\in X$, and by~$B_X(a,d)$ the closed ball of radius $d > 0$ centered at $a\in X$. If $a=0$, we write $B_X(d)$. The distribution of a random variable~$\xi$ is denoted by~$\DD(\xi)$ and the indicator function of a set~$C$ by~$I_C$. 

\smallskip
\noindent
$L^p(D)$ and~$H^s(D)$ denote the usual Lebesgue and Sobolev spaces in a domain $D\subset \R^n$. We use the same notation for spaces of scalar and vector valued functions. The corresponding norms are denoted by~$\|\cdot\|_{L^p}$ and~$\|\cdot\|_s$, respectively.

\smallskip
\noindent
$C_b(X)$ is the space of bounded continuous functions $f:X\to\R$ endowed with the natural norm $\|f\|_\infty=\sup_X|f|$ and $C_+(X)$ is the set of strictly positive functions $f\in C_b(X)$. 

\smallskip
\noindent
$L_b(X)$ stands for the space of functions $f\in C_b(X)$ such that
$$
\|f\|_L:=\|f\|_\infty+\sup_{0<d_X(u,v)\le 1}\frac{|f(u)-f(v)|}{d_X(u,v)}<\infty.
$$
In the case of a compact metric space, we shall drop the subscript~$b$ and write~$C(X)$ and~$L(X)$.  

\smallskip
\noindent
$\BBBBB(X)$ denotes the Borel $\sigma$-algebra on $X$,  $\MM(X)$  the vector  space of signed Borel measures on~$X$ with finite total mass, $\MM_+(X)$  the cone of non-negative measures $\mu\in\MM(X)$, and~$\PP(X)$ the set of  Borel probability measures on $X$.  The vector space~$\MM(X)$ is endowed with the {\it total variation norm\/}
$$
\|\mu\|_{\mathrm{var}}:=\sup_{\Gamma\in\BB(X)}|\mu(\Gamma)|
=\frac12\sup_{\begin{array}{c}
\mbox{\scriptsize $f\in C_b(X)$}\\[-4pt]
\mbox{\scriptsize $\|f\|_\infty\le1$}
\end{array}}\biggl|\int_X f\dd\mu\biggr|. 
$$
When dealing with~$\MM_+(X)$, we also use the {\it Kantorovich--Wasserstein\/} (also called {\it dual-Lipschitz\/}) metric defined by
$$
\|\mu_1-\mu_2\|_L^*
:=\sup_{
\begin{array}{c}
\mbox{\scriptsize $f\in L_b(X)$}\\[-4pt]
\mbox{\scriptsize $\|f\|_L\le1$}
\end{array}
}\biggl|\int_Xf\dd\mu_1-\int_Xf\dd\mu_2\biggr|,
\quad \mu_1,\mu_2\in\MM_+(X).
$$
The topology defined by the Kantorovich--Wasserstein distance coincides with that of weak convergence. We shall write $\mu_n\rightharpoonup\mu$ to denote the weak convergence of~$\{\mu_n\}$ to~$\mu$. 

\smallskip
\noindent
For an integrable function~$f:X\to\R$ and a measure $\mu\in\MM(X)$, we set
$$
\lag f,\mu\rag=\int_Xf(u)\,\mu(\dd u), \quad \|f\|_\mu=\int_X|f(u)|\,\mu(\dd u). 
$$

\smallskip
\noindent
Given a function $f:X\to\R$, we denote by~$f^+$ and~$f^-$ its positive and negative parts, respectively:
$$
f^+=\frac12(|f|+f),\quad f^-=\frac12(|f|-f).
$$

\smallskip
\noindent
Given two Banach spaces $X_1$ and $X_2$, we denote by $L(X_1,X_2)$ the  Banach space of continuous linear operators from~$X_1$ to~$X_2$ with the usual norm. 

\section{The model and the results}
\label{s1}
\subsection{The model}
\label{s1.1}
In this section, we describe a class of discrete-time stochastic systems for which we shall prove the LDP. 
Let $H$ be a real separable Hilbert space with a scalar product~$(\cdot,\cdot)$ and the corresponding norm~$\|\cdot\|$ and let $S:H\to H$ be a continuous mapping. We consider the random dynamical system
\begin{equation} \label{1.1}
u_k=S(u_{k-1})+\eta_k, \quad k\ge1,
\end{equation}
where $\{\eta_k\}$ is a sequence of independent identically distributed (i.i.d.) random variables in~$H$.  System~\eqref{1.1} defines a homogeneous family of Markov chains, and we denote by $P_k(u,\Gamma)$ its transition function and by~$\PPPP_k:C_b(H)\to C_b(H)$ and $\PPPP_k^*:\PP(H)\to\PP(H)$ the corresponding Markov operators. We shall assume that~$S$ satisfies the following three conditions (which are stronger version of those introduced in~\cite{KS-cmp2000}; see also Section~3.2.1 in~\cite{KS-book}). 

\medskip
{\bf (A) Regularity and stability.} 
{\sl The mapping~$S$ is continuously differentiable in the Fr\'echet sense. Moreover, for any $R>r>0$ there are positive constants~$C=C(R)$ and $a=a(R,r)<1$ and an integer $n_0=n_0(R,r)\ge1$ such that
\begin{gather}
\|S(u_1)-S(u_2)\|\le C(R)\|u_1-u_2\|\quad\mbox{for $u_1,u_2\in B_H(R)$}, \label{1.2}\\
\|S^n(u)\|\le \max\{a\|u\|,r\}\quad\mbox{for $u\in B_H(R)$, $n\ge n_0$},
\label{1.3}
\end{gather}
where $S^n$ is the $n^{\mathrm{th}}$ iteration of~$S$.}

\smallskip
Let us denote by~$\KK$ the support of the law for~$\eta_1$ and assume that it is a compact subset in~$H$. Given a closed subset $B\subset H$, define the sequence of sets
$$
\aA_0(B)=B, \quad \aA_k(B)=S(\aA_{k-1}(B))+\KK, \quad k\ge1,
$$
and denote by~$\aA(B)$ the closure in~$H$ of the union of~$\aA_k(B)$. We shall call~$\aA(B)$ the {\it domain of attainability from~$B$\/}. 

\medskip
{\bf (B) Dissipativity.}
{\sl There is $\rho>0$ and a non-decreasing integer-valued function $k_0=k_0(R)$ such that
\begin{equation} \label{1.4}
\aA_k(B_H(R))\subset B_H(\rho)\quad\mbox{for $R\ge0$, $k\ge k_0(R)$}. 
\end{equation}
}

\medskip
{\bf (C) Squeezing.}
{\sl There is an orthonormal basis~$\{e_j\}$ in~$H$ such that, for all $R>0$ and $u_1,u_2\in B_H(R)$,
\begin{equation} \label{1.5}
\|(I-{\mathsf P}_N)(S(u_1)-S(u_2))\|\le \gamma_N(R)\|u_1-u_2\|, 
\end{equation}
where ${\mathsf P}_N:H\to H$ denotes the orthogonal projection on the linear span of~$e_1,\dots,e_N$, and~$\{\gamma_N(R)\}$ is a decreasing sequence of positive numbers 
converging  to zero as $N\to\infty$.}

\smallskip
As for the sequence~$\{\eta_k\}$, we assume that it satisfies the following hypothesis:

\medskip
{\bf (D) Structure of the noise.}
{\sl The random variable~$\eta_k$ has the form
\begin{align}\label{E:eta}
\eta_k=\sum_{j=1}^\infty b_j\xi_{jk}e_j,
\end{align}
where $\{e_j\}$ is the orthonormal basis entering~{\rm(C)}, $b_j\ge0$ are constants such that
\begin{equation} \label{1.7}
\BBBB:=\sum^\infty_{j=1} b_j^2<\infty,
\end{equation}
and~$\xi_{jk}$ are independent scalar random variables. Moreover, the law of~$\xi_{jk} $ is absolutely continuous with respect to the Lebesgue measure, and the corresponding density~$p_j(r)$ is a Lipschitz continuous function such that $p_j (0) > 0$ and $\supp p_j \subset [-1,1]$.} 

\smallskip
Recall that a measure $\mu\in\PP(H)$ is said to be stationary for~\eqref{1.1} if $\PPPP_1^*\mu=\mu$. A proof of the following theorem can be found in Chapter~3 of~\cite{KS-book}. 

\begin{theorem}\label{T:him}
Suppose that Conditions~{\rm(A)--(D)} are fulfilled and that\,\footnote{\,Theorem~\ref{T:him} remains valid if finitely many~$b_j$ are non-zero. However, the main results of this paper on LDP will be proved under the stronger condition~\eqref{1.8}.}
\begin{equation} \label{1.8}
b_j\neq0\quad\mbox{for all $j\ge1$}. 
\end{equation}
Then there is a unique stationary measure $\mu\in\PP(H)$. Moreover, there are  constants $C>0$ and $\alpha>0$ such that, for any $\la\in\PP(H)$, we have
\begin{equation} \label{1.9}
\|\PPPP_k^*\la-\mu\|_L^*
\le Ce^{-\alpha k} \left(1+\int_H \|u\| \la(\dd u)  \right),\quad k\ge0.
\end{equation}
\end{theorem}

We conclude this subsection by a simple remark on the support of the stationary distribution~$\mu$. Let us denote by~$\aA=\aA(\{0\})$ the domain of attainability from zero. Since~$\aA$ is an invariant subset for~\eqref{1.1}, it carries a stationary measure. Since the stationary measure is unique, we must have $\supp\mu\subset\aA$. On the other hand, inequality~\eqref{1.3} and the inclusion $0\in\supp\DD(\eta_1)$ imply that $P_k(u,B_H(r))>0$ for any $u\in\aA$, $r>0$, and $k\gg1$. Combining this fact with the Kolmogorov--Chapman relation, one easily 
 proves that $\supp\mu=\aA$. 

\subsection{The results}
\label{s2.1}
Before formulating the main results of this paper, we  recall some standard definitions from the theory of large deviations (e.g., see Chapter~6 in~\cite{ADOZ00}). Let~$X$ be a Polish space and let~$\PP(X)$ be the space of probability measures on~$X$ endowed with the topology of weak convergence (generated by the Kantorovich--Wasserstein distance). Recall that a {\it random probability measure\/} on~$X$ is defined as a measurable mapping from a probability space~$(\Omega,\FF,\IP)$ to~$\PP(X)$. A mapping $I : \PP(X) \to [0, +\ty]$ is called a {\it rate function\/} if it is lower semicontinuous, and a rate function~$I$ is said to be {\it good\/} if the level set $\{\sigma\in\PP(X) : I(\sigma) \le \alpha \}$ is compact for any $\alpha \in[0, +\infty)$. For a measurable set $\Lambda\subset\PP(X)$, we write $I(\Lambda)=\inf_{\sigma\in\Lambda}I(\sigma)$.

\begin{definition}
Let~$\{\zeta_k=\zeta_k^\omega,k\ge1\}$ be a sequence of random probability measures on~$\aA$. We shall say that $\{\zeta_k\}$ satisfies the LDP with a rate function~$I$ if the following two properties are satisfied. 
\begin{description}
\item[Upper bound.] 
For any closed subset $F\subset\PP(X)$, we have
\begin{align}\label{E:ub}
\limsup_{k\to\infty} \frac1k\log \pP\{\zeta_k\in F\}\le -I(F).
\end{align}
\item[Lower bound.] 
For any open subset $G\subset\PP(X)$, we have
\begin{align}\label{E:lb}
\liminf_{k\to\infty} \frac1k\pP\{\zeta_k\in G\}\ge -I(G).
\end{align}
\end{description}
\end{definition}   

We now consider the family of Markov chains defined by~\eqref{1.1}. To an arbitrary random variable~$u_0$ in~$H$, which we always assume to be independent of~$\{\eta_k\}$,  one associates a   family of {\it occupation measures\/} by the formula
\begin{align}\label{E:1c}
\zeta_k:=  \frac{1}{k}\sum_{n=0}^{k-1} \delta_{u_n},
\end{align}
where $\delta_u$ stands for the Dirac measure concentrated at~$u$. Recall that~$\aA$ denotes the domain of attainability from zero (see the end of Section~\ref{s1.1}). It is a compact subset of~$H$, invariant under the random dynamics defined by~\eqref{1.1}. Note that if the support of~$\DD(u_0)$ is contained in~$\aA$, then~$\zeta_k$ is also supported by~$\aA$. The following  theorem is the main result of this paper. 

\begin{theorem}\label{main} 
Let Hypotheses~{\rm(A)--(D)} and Condition~\eqref{1.8} be fulfilled and let~$u_0$ be an arbitrary random variable in~$H$ whose law is supported by~$\aA$. Then the family~$\{\zeta_k,k\ge1\}$ of random probability measures on~$\aA$ satisfies the LDP with a good rate function~$I$ defined by
\begin{align}\label{E:III} 
I(\sigma)=\sup_{V\in C(\aA)}  \bigl(\lag V,\sigma\rag-Q(V) \bigr),
\quad \sigma\in\PP(\aA),
\end{align}
where  $Q(V)$ is a $1$-Lipschitz convex function such that $Q(C)=C$ for any $C\in\R$. 
\end{theorem}

Note that the lower semi-continuity of~$I$ is built in its definition, while the fact that~$I$ is a good rate function follows from the compactness of~$\PP(\aA)$ in the weak topology. Choosing  suitable open and closed sets in the LDP for occupation measures, we obtain the asymptotics of the time-averages for various functionals of trajectories of~\eqref{1.1}. For instance, let $f:\aA\to\R^m$ be a continuous mapping and let $\Gamma\subset\R^m$ be a Borel set. Define 
$$
F_\Gamma=\{\sigma\in\PP(\aA):\langle f,\sigma\rangle\in \overline\Gamma\}, \quad
G_\Gamma=\{\sigma\in\PP(\aA):\langle f,\sigma\rangle\in \dot \Gamma\},
$$
where $\overline\Gamma$ and $\dot\Gamma$ denote the closure and interior of~$\Gamma$, respectively. In view of the LDP, we have
\begin{align}
\limsup_{k\to\infty}\frac1k\log
\IP\biggl\{\frac1k\sum_{n=0}^{k-1}f(u_n)\in \Gamma\biggr\}
&\le -I(F_\Gamma), \label{ldu}\\
\liminf_{k\to\infty}\frac1k\log
\IP\biggl\{\frac1k\sum_{n=0}^{k-1}f(u_n)\in \Gamma\biggr\}
&\ge -I(G_\Gamma).
\label{ldl} 
\end{align}

Theorem~\ref{main} provides the LDP for the occupation measures~\eqref{E:1c}. Some further analysis combined with the Dawson--G\"artner theorem enables one to derive a process level LDP for trajectories of~\eqref{1.1} issued from~$\aA$. Namely, denote by~$\HHH=H^{\Z_+}$ the direct product of countably many copies of~$H$ and, given a trajectory~$\{u_k\}$ for~\eqref{1.1}, define a sequence of random probability measures on~$\HHH$ by the relation
\begin{equation} \label{rpm3}
\zzeta_k=\frac1k\sum_{n=0}^{k-1}\delta_{\uuu_n},\quad k\ge1,
\end{equation}
where we set $\uuu_n=(u_k,k\ge n)$. 

\begin{theorem}\label{mainty} 
Let the hypotheses of Theorem~\ref{main} be fulfilled and let~$u_0$ be an arbitrary random variable in~$H$ whose law is supported by~$\aA$. Then the family of random probability measures~$\{\zzeta_k,k\ge1\}$ satisfies the LDP with a good rate function~$\III:\PP(\HHH)\to[0,+\infty]$, which is equal to~+$\infty$ outside a compact subset.
\end{theorem}

\subsection{Applications}
\label{s2.2}
\subsubsection*{Two-dimensional Navier--Stokes system}
Let us consider the Navier--Stokes system~\eqref{E:1} in which~$\eta(t,x)$ is a random kick force of the form~\eqref{E:4}. We assume that the kicks~$\{\eta_k\}$ form a sequence of i.i.d.\ random variables in the space~$H$ (see~\eqref{ps}). Normalising the solutions of~\eqref{E:1} to be right continuous and denoting $u_k=u(k,x)$, we see that any solution of~\eqref{E:1} satisfies relation~\eqref{1.1} in which~$S$ stands for the time-$1$ shift along trajectories of the homogeneous Navier--Stokes system (e.g., see Section~2.3 in~\cite{KS-book}). We claim that Theorems~\ref{main} and \ref{mainty} can be applied to~\eqref{1.1} with the above choice of~$S$, provided that we restrict our consideration to trajectories lying in~$\aA=\aA(\{0\})$. Indeed, the differentiability of the flow map for the Navier--Stokes system is well known (e.g., see Section~7.5 in~\cite{BV1992}), and all other properties entering Conditions~(A) and~(B) are checked in~\cite{KS-cmp2000}. Furthermore, the squeezing property~(C) is satisfied for any choice of an orthonormal basis~$\{e_j\}$ in~$H$ (cf.\ proof of Proposition~\ref{NS2} below). We thus obtain the following result. 

\begin{proposition} \label{NS1}
Let the random variables $\{\eta_k\}$ satisfy Condition~{\rm(D)} with $b_j\ne0$ for all $j\ge1$, let~$u_0$ be an arbitrary $H$-valued random variable which is independent of~$\{\eta_k\}$ and whose law is supported by~$\aA$, and let~$\{u_k\}$ be the corresponding trajectory of~\eqref{1.1}. Then the occupation measures~$\zeta_k$ and $\zzeta_k$ defined by~\eqref{E:1c} and \eqref{rpm3} satisfy the LDP with  good rate functions. 
\end{proposition}

 In particular, taking for~$u_0$ a random variable distributed according to the stationary measure~$\mu$, we obtain Theorems~A and~B of the Introduction. Furthermore, in view of the discussion after Theorem~\ref{main}, we have an LDP for the time-averages of continuous functionals $f:H\to\R^m$ calculated on trajectories of~\eqref{1.1}. This result is applicable, for instance, to the energy functional $f(u)=\frac12\int_D|u(x)|^2\dd x$. 

To treat other physically relevant observables, such as the enstrophy or the correlation tensor, we need to change the phase space of the problem, making it more regular. More precisely, let us define the space
$$
U=H\cap H_0^1(D)\cap H^2(D)
$$
(where $H^s(D)$ denotes the usual Sobolev space of order~$s$) and endow it with the usual scalar product in~$H^2$. Since the flow-map for the Navier--Stokes system preserves the $H^2$-regularity, system~\eqref{1.1} can be studied in the space~$U$, provided that the random kicks also 
belong to~$U$. We have the following result. 

\begin{proposition} \label{NS2}
Let~$\{\eta_k\}$ be a sequence of i.i.d.\ random variables in~$U$ of the form~\eqref{E:eta}, in which~$\{e_j\}$ is an orthonormal basis in~$U$, and~$\{b_j\}$ and~$\{\xi_{jk}\}$ are the same as in Condition~{\rm(D)}. Assume that  $b_j\ne0$ for all $j\ge1$. Then the LDP holds for the  occupation measures~$\zeta_k$ and $\zzeta_k$ of any trajectory whose initial state~$u_0$ is a $U$-valued random variable with range in~$\aA$. 
\end{proposition}

For instance, one can take for an initial state any function $u_0\in U$ or a random variable~$u_0$ distributed according to the stationary measure. Furthermore, relations~\eqref{ldu} and~\eqref{ldl} hold for the  functional $f:U\to\R^6$ defined by 
$$
f(u)=\biggl(\frac12\int_D|u(x)|^2\dd x,\frac12\int_D|(\nabla\otimes u)(x)|^2\dd x,u^i(x_1)u^j(x_2), 1\le i,j\le 2\biggr), 
$$
where $u=(u^1,u^2)$, and $x_1,x_2\in D$ are given points.

\begin{proof}[Proof of Proposition~\ref{NS2}]
We shall check that Conditions (A)--(D) of Section~\ref{s1.1} (with~$H$ replaced by~$U$) are fulfilled. The validity of~(D) follows from the hypotheses of the proposition. The facts that, for any $T>0$, the time-$T$ shift along trajectories is uniformly Lipschitz continuous on  bounded subsets of~$U$ and is continuously differentiable in the Fr\'echet sense are proved in Chapter~7 of~\cite{BV1992}. Let us prove~\eqref{1.3}. It is well known that (see the proof of Theorem~6.2 of Chapter~1 in~\cite{BV1992})
$$
\|S(u)\|\le q\,\|u\|, \quad \|S(u)\|_2\le C\,\|u\|,
$$
where $q<1$ and $C>0$ are some constants, $u\in H$ in the first inequality, and $u\in B_H(1)$ in the second. Combining these two estimates we see 
that  for any $R>0$ we can find $n_1=n_1(R)\ge1$ such that
$$
\|S^{n+1}(u)\|_2\le Cq^n\|u\|\quad \mbox{for $u\in B_H(R)$, $n\ge n_1$}.
$$
This inequality immediately implies~\eqref{1.3}. 

\smallskip
We now establish the dissipativity property~(B). We know that this property holds in the space~$H$. Thus, we can find $\rho_1>0$ such that, for any $R>0$ and a sufficiently large integer $k_1(R)\ge1$, we have 
$$
\aA_k(B_U(R))\subset B_H(\rho_1)\quad\mbox{for $k\ge k_1(R)$}. 
$$
Since the mapping~$S$ is continuous from~$H$ to~$U$, it follows that
$$
\aA_{k+1}(B_U(R))\subset S(B_H(\rho_1))+\KK\quad
\mbox{for $k\ge k_1(R)$}.
$$
Choosing $\rho>0$ such  that $S(B_H(\rho_1))+\KK\subset B_U(\rho)$, we obtain~\eqref{1.4} with $H=U$ and $k_0(R)=k_1(R)+1$. 

\smallskip
It remains to prove the squeezing property~(C). Let $u_i(t)$, $i=1,2$, be two solutions of the homogeneous Navier--Stokes system, which we write as a non-local PDE in the space~$H$:
\begin{equation} \label{2.6}
\p_t u+\nu Lu+B(u)=0.
\end{equation}
Here $L=-\Pi\Delta$, $B(u)=B(u,u)$, $B(u,v)=\Pi(\langle u,\nabla\rangle v)$, and~$\Pi$ is the orthogonal projection in~$L^2(D,\R^2)$ onto~$H$. We wish to show that, if the initial conditions $u_{i0}=u_i(0)$ belong to the ball $B_U(R)$ and~$\{e_j\}$ is an orthonormal basis in~$U$, then
$$
\bigl\|(I-{\mathsf P}_N)\bigl(u_1(1)-u_2(1)\bigr)\bigr\|_2
\le \gamma_N(R)\|u_{10}-u_{20}\|_2,
$$
where $\gamma_N(R)$ depends only on the basis~$\{e_j\}$ and goes to zero as $N\to\infty$. A simple compactness argument implies that this inequality will hold if we prove that
\begin{equation} \label{2.7} 
\|u_1(1)-u_2(1)\|_3\le C(R)\|u_{10}-u_{20}\|_2
\quad\mbox{for $u_{10},u_{20}\in B_U(R)$}. 
\end{equation}
The proof of this inequality is carried out by standard methods (e.g., see~\cite{BV1992}), and we confine ourselves to outlining the main steps. We shall denote by~$C_i(R)$ unessential positive constants depending only on~$R$. 

\medskip
{\it Step~1}. 
It suffices to prove that
\begin{align}
\|u_1(1)-u_2(1)\|_2&\le C_1(R)\,\|u_{10}-u_{20}\|_2,\label{2.9}\\
\|\dot u_1(1)-\dot u_2(1)\|_1&\le C_2(R)\,\|u_{10}-u_{20}\|_2
\label{2.10}
\end{align}
for $u_{10},u_{20}\in B_U(R)$, where $\dot v=\p_t v$. Set $u=u_1-u_2$ and note  that ~\eqref{2.6} implies
$$
\nu Lu(1)=-\dot u(1)-B\bigl(u_1(1),u(1)\bigr)-B\bigl(u(1),u_2(1)\bigr).
$$
Since $u_i(1)$, $i=1,2$ are bounded in~$H^2$ (see Theorem~6.2 in~\cite{BV1992}) and the bilinear mapping $(u,v)\mapsto B(u,v)$ is continuous from~$U$ to~$H^1$, we see that
$$
\nu\,\|Lu(1)\|_1\le \|\dot u(1)\|_1+C_3(R)\,\|u(1)\|_2. 
$$
Recalling that $L^{-1}$ is continuous from~$H_0^1\cap H$ to~$H^3$ and using inequalities~\eqref{2.9} and~\eqref{2.10}, we obtain the required estimate~\eqref{2.7}. 

\smallskip
{\it Step~2}. To prove~\eqref{2.9} and~\eqref{2.10}, we first show that 
\begin{equation} \label{2.11}
\sup_{0\le t\le1}\biggl(\|\dot u(t)\|^2+\int_0^t\|\dot u(s)\|_1^2\dd s\biggr)
\le C_4(R)\,\|u_{10}-u_{20}\|_2.
\end{equation}
Differentiating~\eqref{2.6} in time, we derive the following equation for~$\dot u=\dot u_1-\dot u_2$:
\begin{equation} \label{2.12}
\p_t\dot u+\nu L\dot u+B(\dot u_1,u)+B(u_1,\dot u)+B(\dot u,u_2)+B(u,\dot u_2)=0. 
\end{equation}
Taking the $L^2$-scalar product with $2\dot u$ and performing some standard transformations, we obtain
$$
\p_t\|\dot u\|^2+\nu\|\dot u\|_1^2\le C_5\|u_2\|_1^2\|\dot u\|^2
+C_5\|u\|_1\|u\|\,\bigl(\|\dot u_1\|_1\|\dot u_1\|+\|\dot u_2\|_1\|\dot u_2\|\bigr).
$$
Applying the Gronwall and Cauchy--Schwarz inequalities and using the fact that~$\dot u_i$ belong to a bounded set in $L^\infty(0,1;H)\cap L^2(0,1;H^1)$, we derive
\begin{equation} \label{2.13}
\|\dot u(t)\|^2+\nu\int_0^t\|\dot u(s)\|_1^2\dd s
\le C_6(R)\,\Bigl(\|\dot u(0)\|^2+\|u\|_{L^\infty(J_t;H)}\|u\|_{L^2(J_t;H^1)}\Bigr),
\end{equation}
where $J_t=(0,t)$. It follows from~\eqref{2.6} that (cf.\ Step~1)
$$
\|\dot u(0)\|\le C_7(R)\,\|u(0)\|_2.
$$
Furthermore, it is well known that 
\begin{equation} \label{2.17}
\sup_{0\le t\le 1}\biggl(\|u(t)\|^2+t\,\|u(t)\|_1^2+\int_0^t\|u(s)\|_1^2\,\dd s\biggr)\le C_8(R)\,\|u(0)\|^2. 
\end{equation}
Combining these two inequalities with~\eqref{2.13}, we obtain~\eqref{2.11}. 

\smallskip
{\it Step~3}. 
We now derive~\eqref{2.9}. To this end, note that 
\begin{equation} \label{2.16}
\nu Lu(1)=g:=-\p_tu(1)-B(u_1(1),u(1))-B(u(1),u_2(1)). 
\end{equation}
Using the continuity properties of~$B$ and inequalities~\eqref{2.11} and~\eqref{2.17} one easily obtains 
$$
\|g\|\le \|\p_tu(1)\|+C_9(R)\,\|u(1)\|_1\le C_{10}(R)\,\|u(0)\|_2. 
$$
Combining this with~\eqref{2.16} and the continuity of~$L^{-1}$ from~$H$ to~$H^2$, we obtain~\eqref{2.9}. 

\smallskip
{\it Step~4}. 
It remains to prove~\eqref{2.10}. To this end, we take the $L^2$-scalar product of~\eqref{2.12} with $2tL\dot u$. After some standard transformations, we derive
\begin{equation} \label{2.017}
\p_t\bigl(t\|\dot u\|_1^2\bigr)-\|\dot u\|_1^2+2\nu t\,\|\dot u\|_2^2=q(t),
\end{equation}
where we set
$q(t)=-2t\bigl(B(\dot u_1,u)+B(u_1,\dot u)+B(\dot u,u_2)+B(u,\dot u_2),L\dot u\bigr)$. Well-known estimates for the bilinear term~$B$ imply that
\begin{equation} \label{2.018}
|q(t)|\le \nu t\,\|\dot u\|_2^2+C_{11}t\,q_1(t),
\end{equation}
where 
\begin{multline*}
q_1(t)=\|\dot u_1\|\,\|\dot u_1\|_1\|u\|_1\|u\|_2+\|u_1\|^2\|u_1\|_2^2\|\dot u\|^2\\
+\|u_2\|_1\|u_2\|_2\|\dot u\|\,\|\dot u\|_1+\|\dot u_2\|_1^2\|u\|\,\|u\|_2. 
\end{multline*}
Integrating~\eqref{2.017} in time and using~\eqref{2.018}, we obtain
\begin{equation} \label{2.019}
t\,\|\dot u\|_1^2+\nu\int_0^ts\,\|\dot u\|_2^2\,\dd s\le \int_0^t\|\dot u\|_1^2\,\dd s
+C_{11}\int_0^tsq_1(s)\,\dd s. 
\end{equation}
The first integral on the right-hand side can be estimated with the help of~\eqref{2.11}. We now bound each term of the second integral:
\begin{align*}
\int_0^ts\,\|\dot u_1\|\,\|\dot u_1\|_1\|u\|_1\|u\|_2\dd s
&\le C_{12}(R)\,\biggl(\int_0^t\|u(s)\|_2^2\,\dd s\biggr)^{1/2}
\sup_{0\le s\le t}\|u(s)\|_1,\\
\int_0^ts\,\|u_1\|^2\|u_1\|_2^2\|\dot u\|^2\dd s
&\le C_{13}(R)\,\biggl(\int_0^t\|\dot u(s)\|^2\,\dd s\biggr)^{1/2},\\
\int_0^ts\,\|u_2\|_1\|u_2\|_2\|\dot u\|\,\|\dot u\|_1\dd s
&\le C_{14}(R)\,\biggl(\int_0^t\|\dot u(s)\|_1^2\,\dd s\biggr)^{1/2}
\sup_{0\le s\le t}\|\dot u(s)\|,\\
\int_0^ts\,\|\dot u_2\|_1^2\|u\|\,\|u\|_2\dd s
&\le C_{15}(R)\,\biggl(\int_0^t\|u(s)\|_2^2\,\dd s\biggr)^{1/2}
\sup_{0\le s\le t}\|u(s)\|,
\end{align*}
where $0\le t\le 1$, and we used the fact that the functions~$u_i$ and~$\dot u_i$ belong to bounded sets in the spaces $L^\infty(J_1,H^2)$ and $L^\infty(J_1,L^2)\cap L^2(J_1,H^1)$, respectively. 
On the other hand, it is well known that
\begin{equation*}
\sup_{0\le t\le 1}\biggl(\|u(t)\|_1^2+\int_0^t\|u(s)\|_2^2\,\dd s\biggr)
\le C_{17}(R)\,\|u(0)\|_1^2. 
\end{equation*}
Combining these estimates with~\eqref{2.019}, \eqref{2.17}, and~\eqref{2.11}, we obtain~\eqref{2.10}. This completes the proof of Proposition~\ref{NS2}. 
\end{proof}

\subsubsection*{Complex Ginzburg--Landau equation}
Let us consider a complex Ginzburg--Landau (CGL) equation perturbed by a random kick force:
\begin{equation}
\p_t u-(\nu+i)\Delta u+ia|u|^2u=\eta(t,x), \quad x\in D, \quad u\big|_{\p D}=0.\label{2.26}
\end{equation}
Here $a>0$ is a parameter, $D\subset\R^3$ is a bounded domain with $C^2$-smooth boundary~$\p D$, $u=u(t,x)$ is a complex-valued unknown function, and~$\eta$ is an external force of the form~\eqref{E:4}, where $\{\eta_k\}$ is a sequence of i.i.d.\ random variables in the complex space~$H_0^1(D)$. 
It is well known that  the Cauchy problem for~\eqref{2.26} is well posed in~$H_0^1(D)$ (see~\cite{KS-jpa2004}),  that is, for any $u_0\in H_0^1(D)$, problem~\eqref{2.26} has a unique solution such that
\begin{equation} \label{2.28}
u(0,x)=u_0(x). 
\end{equation}
Let us assume that the random kicks entering~\eqref{2.26} have the form
$$
\eta_k(x)=\sum_{j=1}^\infty b_j(\xi_{jk}^1+i\xi_{jk}^2)e_j,
$$
where $\{e_j\}$ is an orthonormal basis in~$H_0^1(D)$ consisting of the real eigenfunctions of the Dirichlet Laplacian, $\{b_j\}\subset \R$ is a sequence satisfying~\eqref{1.7}, and $\xi_{jk}^i$ are independent real-valued random variables whose laws possess the properties stated in Condition~(D) of Section~\ref{1.1}. We denote by~$\aA\subset H_0^1(D)$ the set of attainability from zero. The   following result is an analogue of Propositions~\ref{NS1} and~\ref{NS2} for the case of the CGL equation.

\begin{proposition} \label{CGL}
In addition to the above hypotheses, assume that $b_j\ne0$ for all $j\ge1$. Then the LDP holds for the  occupation measures~\eqref{E:1c} and \eqref{rpm3} of the trajectories whose initial state is  an $H_0^1$-valued random variable with range in~$\aA$.  \end{proposition}

\begin{proof}
We endow the space~$H_0^1=H_0^1(D)$ with the scalar product 
$$
(u_1,u_2)_1=\Re\int_D\nabla u_1(x)\cdot \overline{\nabla u_2(x)}\,\dd x
$$
and regard it as a real Hilbert space. Let $S:H_0^1\to H_0^1$ be the time-$1$ shift along trajectories of the problem~\eqref{2.26} with $\eta\equiv0$. The required results will be established if we check that the stochastic system~\eqref{1.1} considered in the space $H=H_0^1$ possesses properties (A)--(D). Regularity of the mapping~$S$ and its Lipschitz continuity on bounded subsets are standard, and ~(D) is satisfied in view of the hypotheses of the proposition. Thus, it remains to check \eqref{1.3}--\eqref{1.5}. 

\smallskip
{\it Step~1}. 
Let us introduce the following continuous functionals on~$H_0^1$:
$$
\HH_0(u)=\frac12\int_D|u(x)|^2\dd x, \quad 
\HH_1(u)=\int_D\Bigl(\frac12|\nabla u(x)|^2+\frac{a}{4}|u(x)|^4\Bigr)\,\dd x.
$$
It is well known\,\footnote{\,For instance, see Section~2.2 in~\cite{KS-jpa2004} for the more complicated case of a white noise force.} that if~$u(t)$ is a solution of ~\eqref{2.26} and ~$\eta$ is a locally integrable function of time with range in~$H_0^1$, then
\begin{align}
\frac{\dd}{\dd t}\HH_0(u)&=-\nu\,\|\nabla u\|^2+(u,\eta),
\label{2.29}\\
\frac{\dd}{\dd t}\HH_1(u)
&=-\nu\,\|\Delta u\|^2-2a\nu(|u|^2,|\nabla u|^2)+a\nu(u^2,(\nabla u)^2)
+(-\Delta u+a|u|^2u,\eta), 
\label{2.30}
\end{align}
where $(\cdot,\cdot)$ and $\|\cdot\|$ stand for the real $L^2$-scalar product and the corresponding norm:
$$
(u_1,u_2)=\Re\int_Du_1(x)\,\bar u_2(x)\,\dd x, \quad \|u\|^2=(u,u). 
$$ 
Using the inequalities
$$
\|\Delta u\|^2\ge\alpha_1\|\nabla u\|^2, \quad 
|(u^2,(\nabla u)^2)|\le (|u|^2,|\nabla u|^2), \quad
(|u|^2,|\nabla u|^2)\ge c\int_D|u|^4\dd x, 
$$
where $\alpha_1>0$ is the first eigenvalue of the Dirichlet Laplacian,
we derive from~\eqref{2.30} the inequality
\begin{equation} \label{2.031}
\frac{\dd}{\dd t}\HH_1(u(t))
\le-\delta\HH_1(u(t)) -\frac{\nu}{2}\|\Delta u\|^2+(-\Delta u+a|u|^2u,\eta),
\end{equation}
where $\delta>0$ depends only on~$a$, $\nu$, and~$\alpha_1$. 
Taking $\eta\equiv0$ and applying the Gronwall inequality, we see that 
$$
\HH_1(u(t))\le e^{-\delta t}\HH_1(u(0)), \quad t\ge0. 
$$
It follows that if $u_0\in B_{H_0^1}(R)$, then
$$
\|S^n(u_0)\|_1\le\bigl(2\HH_1(u(n))\bigr)^{1/2}
\le\bigl(2e^{-\delta n}\HH_1(u_0)\bigr)^{1/2}\le C_1R\,e^{-\delta n/2}\|u_0\|_1,
$$
where we used the inequality 
\begin{equation} \label{2.032}
\HH_1(v)\le C\|v\|_1^4,\quad v\in H_0^1,
\end{equation}
following immediately from the continuity of the embedding of~$H_0^1\subset L^4$. The above estimate for $\|S^n(u_0)\|_1$ implies~\eqref{1.3}. 

\smallskip
{\it Step~2}. 
We now prove the dissipativity property~\eqref{1.4}. To this end, we first establish a bound for the $L^2$-norm. It follows from~\eqref{2.29} that, for any $\e>0$, the function $\varphi_\e(t)=(\HH_0(u(t))+\e)^{1/2}$ satisfies the inequality
$$
\varphi_\e'(t)\le -\nu\alpha_1\varphi_\e(t)+\tfrac{1}{\sqrt{2}}\|\eta(t)\|+\nu\alpha_1\sqrt{\e}.
$$
Applying the Gronwall inequality and passing to the limit  $\e\to0$, we obtain
\begin{equation} \label{2.033}
\|u(t)\|\le e^{-\nu\alpha_1 t}\|u_0\|
+\int_0^te^{-\nu\alpha_1(t-s)}\|\eta(s)\|\,\dd s. 
\end{equation}
Now note that if $\int_J\|\eta(s)\|\dd s\le b_0$ for any interval $J\subset\R_+$ of length~$1$, then
$$
\int_0^te^{-\nu\alpha_1(t-s)}\|\eta(s)\|\,\dd s\le\frac{b_0}{1-e^{-\nu\alpha_1}}
\quad\mbox{for $t\ge0$}.
$$
Combining this with~\eqref{2.033}, we see that
\begin{equation} \label{2.034}
\|u(k)\|\le e^{-\nu\alpha_1 k}\|u_0\|+\frac{b_0}{1-e^{-\nu\alpha_1}}, \quad
k\ge0. 
\end{equation}
This inequality, established in the case of locally time-integrable functions $\eta(t)$, remains true for kick forces of the form~\eqref{E:4} with $L^2$ bounded functions~$\eta_k$. In particular, \eqref{1.4} holds with $H=L^2(D)$.

\smallskip
We now use the regularising property of the homogeneous CGL equation to prove~\eqref{1.4} with $H=H_0^1$. Namely, if we show that the mapping $S:u_0\mapsto u(1)$ from~$L^2$ to~$H_0^1$ is bounded on bounded subsets, then~\eqref{2.034} will obviously imply the existence of a universal constant $\rho>0$  satisfying~\eqref{1.4} with $H=H_0^1$. To prove the boundedness of~$S$, let us fix a solution $u$ of~\eqref{2.26}, define a function $\psi(t)=t\sqrt{\HH_1(u(t))}$, and calculate its derivative. It follows from~\eqref{2.30} and~\eqref{2.032} that\footnote{\,A rigorous derivation of~\eqref{2.035} can be carried out by the simple argument used to establish~\eqref{2.033}.}
$$
\psi'(t)=\sqrt{\HH_1}+\frac{t}{\sqrt{\HH_1}}\frac{\dd}{\dd t}\HH_1
\le \sqrt{C}\,\|u(t)\|_1^2.
$$
Furthermore, integrating relation~\eqref{2.29} with $\eta\equiv0$, we see that
$$
\|u(t)\|^2+2\nu\int_0^t\|u(s)\|_1^2\dd s\le \|u_0\|^2.
$$
Combining these two relations, we obtain
\begin{equation} \label{2.035}
\|u(t)\|_1\le\sqrt{2\HH_1(u(t))}=\frac{\sqrt2}{t}\psi(t)\le\frac{\sqrt{2C}}{t}\int_0^t\|u\|_1^2\dd s\le \frac{C_2}{\nu t}\|u_0\|^2. 
\end{equation}
The boundedness of~$S$ from~$L^2$ to~$H_0^1$ is a straightforward consequence of this inequality. 

\smallskip
{\it Step~3}. 
It remains to establish the squeezing property~\eqref{1.5}, in which~${\mathsf P}_N$ is the orthogonal projection in~$H_0^1$ (endowed with the scalar product~$(\cdot,\cdot)_1$) to the vector span of~$\{e_j,ie_j,j=1,\dots,N\}$. Let us set ${\mathsf Q}_N=I-{\mathsf P}_N$. By hypothesis, we have $\|e_j\|_1=1$, hence it follows that
$$
\|e_j\|=\frac{1}{\sqrt{\alpha_j}},\quad j\ge1,
$$
where $\alpha_j$ is the eigenvalue of the Dirichlet Laplacian corresponding to the eigenfunction~$e_j$. Using this relation, it is straightforward to check that $\{\sqrt{\alpha_j}\,e_j\}$ is an orthonormal basis in~$L^2$ and that the norm of~${\mathsf Q}_N$ regarded as an operator in~$L^2$ is equal to~$1$. 

Now let $u_1, u_2$ be two solutions of~\eqref{2.26} corresponding to initial data $u_{10},u_{20}\in B_{H_0^1}(R)$. Applying~${\mathsf Q}_N$ to~\eqref{2.26} and setting $w={\mathsf Q}_N(u_1-u_2)$, we derive the equation
$$
\dot w-(\nu+i)\Delta w+ia{\mathsf Q}_N\bigl(|u_1|^2u_1-|u_2|^2u_2\bigr)=0. 
$$
It follows that
\begin{align}
\p_t\|w\|_1^2&=2\Re\int_D\nabla\dot w\cdot\nabla\bar w\,\dd x
=-2\Re\int_D\dot w\,\Delta\bar w\,\dd x\notag\\
&=-2\bigl((\nu+i)\Delta w-ia{\mathsf Q}_N(|u_1|^2u_1-|u_2|^2u_2),\Delta\bar w\bigr)\notag\\
&\le -2\nu\,\|\Delta w\|^2+a\,\bigl\||u_1|^2u_1-|u_2|^2u_2\bigr\|\,\|\Delta w\|,
\label{2.036}
\end{align}
where we used the fact that the norm of ${\mathsf Q}_N$ is equal to~$1$. 
Using the H\"older inequality and the continuity of the embedding $H_0^1\subset L^6$, we derive
$$
\bigl\||u_1|^2u_1-|u_2|^2u_2\bigr\|
\le C_3\bigl(\|u_1\|_1+\|u_2\|_1\bigr)^2\,\|w\|_1.
$$
Substituting this into~\eqref{2.036} and using the Poincar\'e inequality 
$\|\Delta w\|^2\ge \alpha_N\|w\|_1^2$, we obtain
\begin{equation} \label{2.037}
\p_t\|w\|_1^2\le 
-\bigl(\nu\alpha_N-C_4(\|u_1\|_1+\|u_2\|_1)^4\bigr)\,\|w\|_1^2.
\end{equation}
Since $u_{10},u_{20}\in B_{H_0^1}(R)$ and the resolving operator for the CGL is bounded on bounded subsets, we can find $C_5(R)$ such that
$$
\|u_i(t)\|_1\le C_5(R)\quad\mbox{for $0\le t\le 1$, $i=1,2$}.
$$
Combining this with~\eqref{2.037} and the Gronwall inequality, we derive
\begin{align*}
\|w(t)\|_1^2
&\le \exp\biggl(-\nu\alpha_Nt
+C_4\int_0^t(\|u_1\|_1+\|u_2\|_1)^4\dd s\biggr)\,\|w(0)\|_1^2\\
&\le \exp\bigl(-\nu\alpha_Nt+C_6(R)\,t\bigr)\,\|w(0)\|_1^2. 
\end{align*}
Since ${\mathsf Q}_N$ is an orthogonal projection in~$H_0^1$, we have $\|w(0)\|_1\le \|u_{10}-u_{20}\|_1$. Substituting this into the above estimate and taking $t=1$, we obtain
$$
\|u_1(1)-u_2(1)\|_1\le \gamma_N(R)\,\|u_{10}-u_{20}\|_1, \quad 
\gamma_N^2(R)=\exp\bigl(-\nu\alpha_N+C_6(R)\bigr). 
$$
This completes the proof of~\eqref{1.5} and  Proposition~\ref{CGL} follows.
\end{proof}

\subsection{Scheme of the proof of Theorem~\ref{main}}
\label{s2.3} 
Along with~$\zeta_k$, let us consider ``shifted'' occupation measures defined as
$$
\hat \zeta_k=\frac1k\sum_{n=1}^k\delta_{u_n}. 
$$
The sequences~$\{\zeta_k\}$ and~$\{\hat\zeta_k\}$ are exponentially equivalent (see Lemma~\ref{l5.2}), and therefore, by Theorem~4.2.13 of~\cite{ADOZ00}, it suffices to prove the LDP for~$\hat\zeta_k$. The proof of this property is based on an abstract result established by Kifer~\cite{kifer-1990}. For the reader's convenience, its statement is recalled in the Appendix (see Theorem~\ref{T:Kif}). We shall prove that the following two properties hold.

\begin{description}
\item[\bf  Property 1: The existence of a limit.]  For any $V\in C(\aA)$, the limit 
\begin{align}\label{E:3.1*}
Q(V)=\lim_{k\to+\ty} \frac{1}{k}
\log\E\,\exp\biggl(\,\sum_{n=1}^kV(u_n)\biggr).
\end{align} 
exists and does not depend on the initial condition~$u_0$.
\end{description}
Let us denote by~$I:\MM(\aA)\to\R_+$ the Legendre transform of~$Q(V)$; see~\eqref{E:H4}. It is well known that 
$$
Q(V)=\sup_{\sigma\in\PP(\aA)}\bigl(\langle V,\sigma\rangle-I(\sigma)\bigr);
$$
see Lemma~2.2 in~\cite{BD99} and Theorem~2.2.15 in~\cite{DS1989}. In view of the compactness of~$\PP(\aA)$, for any $V\in C(\aA)$ the supremum in the above relation is attained at some point~$\sigma_V\in\PP(\aA)$. Any such point is called an {\it equilibrium state\/}.  

\begin{description}
\item[\bf Property 2: Uniqueness of the equilibrium state.]
There is  a dense vector space  $\VV\subset C(\aA)$ such that, for any $V\in\VV$, there exists unique ~$\sigma_V$ satisfying:
\begin{align}\label{E:3.1**}
Q(V)=  \lag V, \sigma_V\rag-I(\sigma_V).
\end{align} 
\end{description}

According to Kifer's theorem, the first of the above properties implies the LD upper bound for~$\hat\zeta_k$, while the second is sufficient for the LD lower bound. The proofs of these two  properties are related to the large-time behaviour of a generalised Markov semigroup associated with~$u_k.$ More precisely, given a function $V\in C(\aA)$, we consider the semigroup
\begin{equation} \label{2.18}
\PPPP_k^V f(u):= 
\E_{u}f(u_k)\exp \biggl(\,\sum_{n=1}^kV(u_n)\biggr), 
\quad f\in C(\aA),
\end{equation}
where the subscript~$u$ means that we consider the trajectory of~\eqref{1.1} starting  from~$u\in H$.  The dual semigroup 
is denoted by 
$\PPPP_k^{V*} :\PP(\aA)\to \PP(\aA)$. We construct explicitly a dense vector space   $\VV\subset C(\aA)$ such that, for any $V\in\VV$,  the semigroup 
$\PPPP_k^V$ is uniformly Feller and uniformly irreducible (see Section~\ref{s3} for the definition of these concepts). Then, by an abstract result proved in Section~\ref{s3}, there is a number $\lambda_V>0$, a function $h_V\in C_+(\aA)$, and a measure $\mu_V\in\PP(\aA)$ satisfying
\begin{equation} \label{2.15}
\PPPP_1^V h_V=\la_V h_V, \quad 
\PPPP_1^{V*} \mu_V=\la_V\mu_V,
\end{equation}
such that  for any $f\in  C(\aA)$ and $\nu\in \PP(\aA)$ we have
\begin{align}
\la_V^{-k}\PPPP_k^V f&\to \lag f,\mu_V\rag h_V \quad \text{in $C(\aA)$ as $k\to+\ty$},\label{c1}\\
\la_V^{-k}\PPPP_k^{V*} \nu& \rightharpoonup \lag h_V,\nu\rag \mu_V\quad \text{in~$\PP(\aA)$ as $k\to+\ty$}.\label{c2}
\end{align} 
Taking $f=1$ in~\eqref{c1}, one gets immediately the existence of the  limit~\eqref{E:3.1*} for $V\in\VV$ and any initial function~$u_0$ whose law is supported by~$\aA$. Then, by a simple approximation argument, we prove the existence of the limit  for any $V\in C(\aA)$.
 
To establish the uniqueness of $\sigma_V\in \PP(\aA)$ satisfying~\eqref{E:3.1**}, we first show that any equilibrium state~$\sigma_V$ is a stationary measure for the following Markov semigroup:
\begin{equation} \label{2.22}
\SSS^V_k g:=\la_V^{-k} h_V^{-1} \PPPP_k^V (gh_V), 
\quad g\in C(\aA).
\end{equation}
We then deduce the uniqueness of stationary measure for~$\SSS^V_k$ from convergence~\eqref{c2}, showing that $\sigma_V(\dd u)= h_V(u) \mu_V (\dd u)$. 

\smallskip
The crucial point in the realisation of the above scheme is the verification of the uniform Feller property for the semigroup~$\{\PPPP_k^V\}$. 
This verification is based  on the Lyapunov--Schmidt reduction and is carried out in Section~\ref{sec-unif}.  

\subsection{Scheme of the proof of Theorem~\ref{mainty}}
\label{s2.4} 
Let $p_m:\HHH\to H^m$ be the projection that maps a sequence $(u_j,j\in\Z_+)$ to the vector $(u_j,0\le j\le m-1)$. It is straightforward to check that if $\{u_k,k\ge0\}$ is a trajectory for~\eqref{1.1}, then the image of~$\zzeta_k$ (see~\eqref{rpm3}) under~$p_m$ coincides with the random probability measure
\begin{equation} \label{rpm2}
\zeta_k^m=\frac1k\sum_{n=0}^{k-1}\delta_{\uu_n^m},\quad k\ge1,
\end{equation}
where $\uu^m_n=(u_n,\ldots,u_{n+m-1})$. It follows from the Dawson--G\"artner theorem (see Theorem~\ref{t5.2} in the Appendix) that  to prove Theorem~\ref{mainty} it suffices to show that  for any integer~$m\ge1$, the LDP holds for~$\zeta_k^m$ with a good rate function~$I_m:\PP(H^m)\to[0,+\infty]$. The 
proof of this fact is very similar to the proof of Theorem \ref{main}  and the argument is outlined in Section \ref{s6}. 
To formulate the result precisely, let~$\aA^{(m)}$ be the set of vectors 
$(u_1,\dots,u_m)\in H^m$ such that $u_1\in\aA$ and $u_k=S(u_{k-1})+\eta_k$ for $2\le k\le m$, where~$\eta_k\in\KK$. 
Note that if a trajectory~$\{u_k\}$ for~\eqref{1.1} is such that~$u_0$ is an $\aA$-valued random variable, then the measures~$\zeta_k^m$ are concentrated on~$\aA^{(m)}$. In Section \ref{s6} we prove 

\begin{theorem}\label{mainm} 
Under the conditions of Theorem~\ref{main}, let~$u_0$ be a random variable in~$H$ whose law is supported by~$\aA$. Then the family~$\{\zeta_k^m,k\ge1\}$ regarded as a sequence of random probability measures on~$\aA^{(m)}$ satisfies the LDP with a good rate function~$I_m:\PP(\aA^{(m)})\to[0,+\infty]$. Moreover, $I_m$~can be written as
\begin{align}\label{E:III^m} 
I_m(\sigma)=\sup_{V\in C(\aA^{(m)})} 
\bigl(\langle V,\sigma\rangle-Q_m(V)\bigr),
\quad \sigma\in\PP(\aA^{(m)}),
\end{align}
where $Q_m:\aA^{(m)}\to\R$ is a $1$-Lipschitz convex function such that $Q_m(C)=C$ for any $C\in\R$. 
\end{theorem}

This result immediately implies that~$\zeta_k^m$, as measures  on ~$H^m$, satisfy  the LDP. To see this,  extend the rate function~$I_m$ constructed in Theorem~\ref{mainm} to the space~$\PP(H^m)$ by setting $I_m(\sigma)=+\infty$ for any measure $\sigma\in\PP(H^m)$ satisfying  $\sigma(\aA^{(m)})<1$. Then, recalling that~$\zeta_k^m$ are supported on~$\aA^{(m)}$ if so is the the initial measure~$\DD(u_0)$, we readily check that the LD upper and lower bounds hold for the family~$\{\zeta_k^m,k\ge1\}$ regarded as random probability measures on~$H^m$. 

\subsection{Uniform large deviations principle}
\label{s1.6}
The arguments of the proofs of Theorems~\ref{main} and~\ref{mainty} enable one to obtain a {\it uniform LDP\/} for the families~$\{\zeta_k\}$ and~$\{\zzeta_k\}$, which depend on the initial point. More precisely, let us denote by~$\zeta_k(u)$ the occupation measure~\eqref{E:1c} for the trajectory issued from a deterministic point $u\in\aA$ and define~$\zzeta_k(u)$ in a similar way. 
The definition of the uniform LDP is recalled in the Appendix (see Section~\ref{s5.2}). We have the following result. 

\begin{theorem}\label{T:umain}
Let Hypotheses~{\rm(A)--(D)} and Condition~\eqref{1.8} be satisfied. Then the uniform LDP holds for the families~$\{\zeta_k(u),u\in\aA\}$ and $\{\zzeta_k(u),u\in\aA\}$ with the good rate functions~$I$ and~$\III$ defined in Theorems~\ref{main} and~\ref{mainty}, respectively. 
\end{theorem}

\begin{proof}[Sketch of the proof] 
Let us define the set $\Theta:=\N\times \aA$ and introduce an order relation~$\prec$ on it by the following rule:  if $\theta_i=(k_i,u^i)\in \Theta$ for $i=1,2$, then $\theta_1\prec\theta_2$ if and only if $k_1\le k_2$. Then~$(\Theta, \prec)$ is a directed set. Defining $r(\theta)=k$, we apply Theorem~\ref{T:Kif} to the family $\zeta_\theta=\zeta_k(u)$ indexed by $\theta=(k,u)\in \Theta$. The scheme of the proof described above for Theorem~\ref{main} applies equally well in this case, and using the fact that the convergence in~\eqref{E:3.1*} is uniform with respect to  the deterministic initial condition $u_0\in \aA$, we get the existence of limit~\eqref{E:H1} and uniqueness of equilibrium measure. Thus, we have the LDP
\begin{equation} \label{1.49}
-I(\dot\Gamma)
\le \liminf_{\theta\in \Theta}\frac1k\log   \IP\{\zeta_\theta\in\Gamma\}
\le \limsup_{\theta\in \Theta}\frac1k\log   \IP\{\zeta_\theta\in\Gamma\}
\le -I(\overline\Gamma).
\end{equation}
Now notice that the middle terms in this inequality can be written as
\begin{align*}
\liminf_{\theta\in \Theta}\frac1k\log\IP\{\zeta_\theta\in\Gamma\}&=\liminf_{k\ri+\ty}\frac1k \inf_{u\in\aA}\log\IP\{\zeta_k(u)\in\Gamma\},
\\
\limsup_{\theta\in \Theta}\frac1k\log\IP\{\zeta_\theta\in\Gamma\}&=\limsup_{k\ri+\ty}\frac1k \sup_{u\in\aA}\log\IP\{\zeta_k(u)\in\Gamma\}.
\end{align*}
Substituting these relations into~\eqref{1.49}, we obtain the uniform LDP for~$\zeta_k(u)$. 

\smallskip
To establish the uniform LDP for~$\zzeta_k(u)$, we apply Theorem~\ref{t6.4}. We thus need the uniform LDP for the projected measures~$\zeta_k^m=\zeta_k^m(u)$ defined in Section~\ref{s2.4}. The latter can be obtained by modifying the proof of Theorem~\ref{mainm} exactly in the same way as we did above to get the uniform LDP for~$\zeta_k(u)$. 
\end{proof}

\section{Large-time asymptotics for generalised Markov semigroups}
\label{s3}
In this section, we prove a general result on the large-time behaviour of trajectories for a class of dual semigroups. This type of results were established earlier for Markov semigroups satisfying a uniform Feller and an irreducibility properties; see~\cite{LY-1994,szarek-1997,KS-cmp2000,LS-2006,KS-book}. The main theorem of this section is a generalisation of Theorem~4.2 in~\cite{KS-cmp2000} and has independent interest.

\medskip
Let $X$ be a compact metric space, let $\MM_+(X)$ be the space of non-negative Borel measures on~$X$ endowed with the topology of weak convergence,  and let $\{P(u,\cdot),u\in X\}\subset\MM_+(X)$ be a family satisfying the following condition:
\begin{description}
\item[Feller property.]
The function $u\mapsto P(u,\cdot)$  from~$X$ to~$\MM_+(X)$ is continuous and non-vanishing. 
\end{description}
In this case, we shall say that $P(u,\Gamma)$ is a {\it generalised Markov kernel\/}. One obvious consequence of  the Feller  property is the inequality
\[
C^{-1}\le P(u,X)\le C\quad\mbox{for all $u\in X$.}
\]
Define the operators
$$
\PPPP f(u)=\int_XP(u,dv)f(v), \quad 
\PPPP^*\mu(\Gamma)=\int_XP(u,\Gamma)\mu(du)
$$
and denote $\PPPP_k=\PPPP^k$ and $\PPPP_k^*=(\PPPP^*)^k$. It is easy  to see that
$$
\PPPP_kf(u)=\int_XP_k(u,dv)f(v), \quad 
\PPPP_k^*\mu(\Gamma)=\int_XP_k(u,\Gamma)\mu(du), 
$$
where $P_k(u,\Gamma)$ is defined by the relations $P_0(u,\cdot)=\delta_u$,  
$P_1(u,\cdot)=P(u,\cdot)$, and 
$$
P_k(u,\cdot)=\int_XP_{k-1}(u,dv)P(v,\cdot),\quad k\ge2. 
$$
To simplify the notation, the sup-norm on~$C(X)$ is denoted in this section by~$\|\cdot\|$. Let~$\mathbf 1$ be the function on~$X$ identically equal to~$1$. 
Recall that a family~$\CC\subset C(X)$ is called  {\it determining\/} if any two measures $\mu,\nu\in\MM_+(X)$ satisfying the relation $\lag f,\mu\rag=\lag f,\nu\rag$ for all $f\in\CC$ coincide.  In this section we prove:

\begin{theorem} \label{t3.1}
Let $P(u,\Gamma)$ be a generalised Markov kernel satisfying the following conditions.
\begin{description}
\item[Uniform Feller property.]
There is a determining family $\CC\subset C_+(X)$ of non-zero functions such that ${\mathbf1}\in\CC$ and the sequence $\{\|\PPPP_kf\|^{-1}\PPPP_kf,k\ge0\}$ is equicontinuous for any $f\in \CC$. 
\item[Uniform irreducibility.]
For any $r>0$ there is an integer $m\ge1$ and a constant $p>0$ such that
\begin{equation} \label{9}
P_m(u,B(\hat u,r))\ge p\quad\mbox{for all $u,\hat u\in X$}. 
\end{equation}
\end{description}
Then there is a constant $\lambda>0$, a unique measure $\mu\in\PP(X)$ whose support coincides with~$X$, and a unique  $h\in C_+(X)$ satisfying $\lag h,\mu\rag=1$, such that  for any $f\in C(X)$ and $\nu\in\MM_+(X)$ we have
\begin{gather}
\PPPP h=\lambda h, \quad \PPPP^*\mu=\lambda\mu,\label{10}\\
\lambda^{-k}\PPPP_k f\to\lag f,\mu\rag h\quad\mbox{in $C(X)$ as $k\to\infty$}, \label{11}\\
\lambda^{-k}\PPPP_k^*\nu\rightharpoonup\lag h,\nu\rag\mu
\quad\mbox{as $k\to\infty$}. \label{12}
\end{gather}
 \end{theorem}

\begin{proof} Note that the  uniqueness of $h$ and $\mu$ is an immediate consequence of the normalisation and relations (\ref{10})-(\ref{12}). 
We split the proof in four steps.

{\it Step~1}. We first prove the existence of a measure satisfying the second relation in~\eqref{10}. To this end, let 
 $F:\PP(X)\to\PP(X)$ be a map defined by 
 \[F(\mu)=(\PPPP^*\mu(X))^{-1}\PPPP^*\mu.
 \]
The Feller property implies that this  map is well defined and  continuous in the weak topology. 
Since $\PP(X)$ is a convex compact set, by the Leray--Schauder theorem, the mapping~$F$ has a fixed point~$\mu\in\PP(X)$. We thus obtain the second relation in~\eqref{10} with $\lambda=\PPPP^*\mu(X)$. In what follows, we may assume without loss of generality that $\lambda=1$; otherwise, we can replace~$P(u,\Gamma)$ by $\lambda^{-1}P(u,\Gamma)$. 

\smallskip
{\it Step~2}. 
Let us prove that, for any $f\in\CC$, we have
\begin{equation} \label{15}
C_f^{-1}\le \|\PPPP_{k}f\|\le C_f\quad\mbox{for all $k\ge1$},
\end{equation}
where $C_f>1$ is a constant not depending on~$k$. Indeed, suppose that there is a sequence $k_j\to\infty$ such that 
\begin{equation} \label{16}
\|\PPPP_{k_j}f\|+\|\PPPP_{k_j}f\|^{-1}\to+\infty\quad\mbox{as $j\to\infty$}.
\end{equation}
In view of the uniform Feller property, we can assume that 
$$
\|\PPPP_{k_j}f\|^{-1}\PPPP_{k_j}f\to g
\quad\mbox{in $C(X)$ as $j\to\infty$},
$$
where $g\in C(X)$ is function whose norm is equal to~$1$. 
Integrating with respect to~$\mu$ and using the invariance of~$\mu$, we derive
\begin{equation} \label{14}
\|\PPPP_{k_j}f\|^{-1}\lag f,\mu\rag\to\lag g,\mu\rag\quad\mbox{as $j\to\infty$}.
\end{equation}
The uniform irreducibility  implies that for any $\hat u\in X$ and $r>0$ we have
$$
\mu\bigl(B(\hat u,r)\bigr)=\int_XP_m\bigl(u,B(\hat u,r)\bigr)\mu(du)
\ge p\,\mu(X)>0.
$$ 
Hence, $\supp\mu=X$, and since $f,g\in C_+(X)$ are non-zero functions, we have that $\lag f,\mu\rag>0$ and $\lag g,\mu\rag>0$. It now follows from~\eqref{14} that the sequence $\|\PPPP_{k_j}f\|$ has a finite positive limit, and therefore~\eqref{16} cannot hold. 

\smallskip
{\it Step~3}. 
Let us prove the existence of~$h\in C_+(X)$ satisfying the first relation in~\eqref{10} with $\lambda=1$. Let $f\in\CC$ be an arbitrary function. The uniform Feller property and inequality~\eqref{15} imply that the sequence $\PPPP_kf$ is uniformly equicontinuous. It follows that so is the sequence 
$$
f_k=\frac1k\sum_{l=0}^{k-1}\PPPP_lf. 
$$
Let $h$ be a limit point for~$\{f_k\}$. It is straightforward to see that $h\ge0$ and $\PPPP_1h=h$. Furthermore, since $\lag f_k,\mu\rag =\lag f,\mu\rag>0$, we see that~$h$ is non-zero. Multiplying~$h$ by a constant, we can assume that $\lag h,\mu\rag=1$. It remains to prove that $h(u)>0$ for all $u\in X$. Indeed, let $\hat u\in X$ be any point at which~$h$ is positive. Then there is $r>0$ such that
$h(v)\ge r$ for $v\in B(\hat u,r)$. It follows that, for any $u\in X$, we have
\begin{align*}
h(u)=\PPPP_mh(u)&=\int_XP_m(u,dv)h(v)\ge\int_{B(\hat u,r)}P_m(u,dv)h(v)\\
&\ge rP_m\bigl(u,B(\hat u,r)\bigr)\ge rp>0,
\end{align*}
where $m\ge1$ is the integer from~\eqref{9}.

\smallskip
{\it Step~4}. 
We now establish convergence~\eqref{11} and~\eqref{12} with $\lambda=1$. To this end, we first note that~\eqref{12} is an immediate consequence of~\eqref{11}. Furthermore, the right-hand inequality in~\eqref{15} with $f={\mathbf1}$ implies that the norms of the operators~$\PPPP_k$ are bounded by~$C_{\mathbf1}$ for all $k\ge1$. Since the linear span of a determining family is dense in~$C(X)$, it suffices to establish~\eqref{11} for any $f\in\CC$. 

Let us fix an arbitrary $f\in\CC$ and define the function $g=f-\lag f,\mu\rag h$. We need to prove that $\PPPP_kg\to0$ in~$C(X)$. Since $\{\PPPP_kg,k\ge0\}$ is uniformly equicontinuous and the norms of~$\PPPP_k$ are bounded, the required assertion will be established if we prove that any sequence of integers~$n_i \rightarrow \infty$ contains a subsequence $\{k_j\}\subset\{n_i\}$ for which
\begin{equation} \label{18}
\|g_{k_j}\|_\mu\to0\quad\mbox{as $j\to\infty$},
\end{equation}
where we set $g_k=\PPPP_kg$. 
Since $\lag g_k,\mu\rag=0$ for $k\ge0$, convergence~\eqref{18} certainly holds for any subsequence $\{k_j\}$ such that $\|g_{k_j}^+\|\to0$ or $\|g_{k_j}^-\|\to0$ as $j\to\infty$. Let us assume that there is no subsequence satisfying this property. Then there exists  sequences~$\{u_i^\pm\}\subset X$ and a constant~$\alpha>0$ such that
\begin{equation} \label{17}
\tilde g_i^+(u_i^+)=\max_{u\in X}\tilde g_i^+(u)\ge\alpha, \quad
\tilde g_i^-(u_i^-)=\max_{u\in X}\tilde g_i^-(u)\ge\alpha, 
\end{equation}
where we set $\tilde g_i=g_{n_i}$. Since $\tilde g_i^\pm$ are uniformly equicontinuous, we can find $r>0$ not depending on~$k$ such that
\begin{equation} \label{018}
\tilde g_i^\pm(u)\ge \frac12 \tilde g_i^\pm(u_i^\pm)\quad\mbox{for $u\in B(u_i^\pm,r)$}. 
\end{equation}
Let~$m$ and~$p$ be the constants arising in the uniform irreducibility condition. 
Then~\eqref{018} and~\eqref{15} imply
\begin{align*}
\PPPP_m\tilde g_i^\pm(u)&=\int_XP_m(u,dv)\tilde g_i^\pm(v)
\le C_{\mathbf1}\tilde g_i^\pm(u_i^\pm),\\
\PPPP_m\tilde g_i^\pm(u)&\ge\int_{B(u_i^\pm,r)}P_m(u,dv)\tilde g_i^\pm(v)
\ge p\,\tilde g_i^\pm(u_i^\pm)/2,
\end{align*}
and  it follows that 
\begin{equation} \label{19}
\sup_{u\in X}\PPPP_m\tilde g_i^\pm(u)\le A_g\inf_{u\in X}\PPPP_m\tilde g_i^\pm(u),
\end{equation}
where $A_g=2C_{\mathbf1}/p$. We now write
\begin{align}
\|\PPPP_m\tilde g_i\|_\mu
&=\int_X|\PPPP_m(\tilde g_i^+-\tilde g_i^-)|d\mu\notag\\
&=\int_X\bigl|(\PPPP_m\tilde g_i^+-A_g^{-1}\|\tilde g_i^+\|_\mu)-
(\PPPP_m\tilde g_i^--A_g^{-1}\|\tilde g_i^-\|_\mu)\bigr|d\mu\notag\\
&\le \int_X\bigl|\PPPP_m\tilde g_i^+-A_g^{-1}\|\tilde g_i^+\|_\mu\bigr|d\mu+
\int_X\bigl|\PPPP_m\tilde g_i^--A_g^{-1}\|\tilde g_i^-\|_\mu\bigr|d\mu\notag\\
&=\int_X\PPPP_m(\tilde g_i^++\tilde g_i^-)d\mu-A_g^{-1}\bigl(\|\tilde g_i^+\|_\mu+\|\tilde g_i^-\|_\mu\bigr)\notag\\
&=(1-A_g^{-1})\|\tilde g_i\|_\mu.\label{22} 
\end{align}
Furthermore, for any $f\in C(X)$ and $k\ge1$, we have
$$
\|\PPPP_kf\|_\mu=\langle|\PPPP_kf|,\mu\rangle\le\langle\PPPP_k|f|,\mu\rangle
=\langle|f|,\mu\rangle=\|f\|_\mu. 
$$
It follows that the sequence $\{\|\PPPP_kg\|_\mu\}$ is non-increasing. Combining this property with~\eqref{22}, we see that if $n_l\ge n_i+m$, then
\begin{equation} \label{20}
\|g_{n_l}\|_\mu\le (1-A_g^{-1})\|g_{n_i}\|_\mu.
\end{equation}
Let us choose a subsequence $\{k_j\}\subset\{n_i\}$ such that $k_{j+1}\ge k_j+m$. Then~\eqref{20} implies that
$$
\|g_{k_j}\|_\mu=\|\PPPP_{k_j}g\|_\mu
\le (1-A_g^{-1})^j\|g\|_\mu\quad\mbox{for $j\ge0$},
$$
where $k_0=0$. This proves convergence~\eqref{18} and  completes the proof of the theorem. 
\end{proof}

\section{The uniform Feller property}
\label{sec-unif}
We shall use freely  the notation introduced in Subsection \ref{s1.1} (we recall  in particular that 
$\{e_j\}$ is the orthonormal basis introduced in Condition~(C), that ~${\mathsf P}_N$ is   the orthogonal projection onto $H_N=\lspan\{e_1,\dots,e_N\}$, and 
that ${\cal A}={\cal A}(\{0\})$ is the domain of attainability from zero).  Let~$\VV$ be the set  of functions $V\in C(\aA)$ for which there is an integer $N\ge1$ and a function $F\in C(H_N)$ such that
\begin{equation} \label{4.4}
V(u)=F({\mathsf P}_Nu)\quad\mbox{for $u\in\aA$}.
\end{equation}
It is easy to see that~$\VV$ is a dense subspace in~$C(\aA)$ containing the constant functions. In particular, the intersection $\CC=\VV\cap C_+(\aA)$ is a determining family for~$\PP(\aA)$. 

For any $V\in C(\aA)$, let us consider the following generalised Markov kernel on~$\aA$:
\begin{align}\label{E:3.2}
P^V_1(u,\Gamma)=\E_u \bigl(I_\Gamma(u_1) e^{V(u_1)}\bigr)
=\int_\Gamma P_1(u,dv)e^{V(v)},\quad u\in \aA, \quad
\Gamma\in \BB(\aA).
\end{align}
The corresponding semigroup of operators is given by~\eqref{2.18}. The goal of this section is to prove:

\begin{theorem} \label{main-tech}
Under the hypotheses of Theorem~\ref{main}, for any $V\in\VV$ the semigroup $\{\PPPP_k^V\}$ possesses the uniform Feller property for the determining class~$\CC$. In other words, for any $V\in\VV$ and $f\in\CC$ the sequence $\{\|\PPPP_k^Vf\|_\infty^{-1}\PPPP_k^Vf,k\ge0\}$ is uniformly equicontinuous. 
\end{theorem}
The theorem  will play a key role in the proof of our main results.  
We start its  proof by discussing a Lyapunov--Schmidt type reduction which gives a useful formula for the Markov semigroup applied to functions depending on finitely many Fourier modes and an auxiliary inhomogeneous Markov chain that will be needed in what follows. 

\subsection{Lyapunov--Schmidt reduction}
\label{s1.2}
In this subsection we introduce a Markov family which is defined in the space of sequences and is closely related to system~\eqref{1.1}. Recall that\,\footnote{\,The space~$\MH_N$ is, of course, the same as ${\boldsymbol H}=H^{\Z_0}$. We use, however,  different notation because the components lying in~$H_N$ and~$H_N^\bot$ will have different meanings.}
$$
 \MH_N=\bigl(H_N\times H_N^\bot\bigr)^{\Z_0},
$$
where~$\Z_0$ the set of non-positive integers.
The set~$\MH_N$ is endowed with the Tikhonov topology and the corresponding Borel $\sigma$-algebra. Given a measurable mapping $T_0:\MH_N\to H_N$, define a family of Markov chains in~$\MH_N$ by the relations
\begin{equation} \label{1.11}
 \bg^0=\boldsymbol U, \quad 
 \bg^k=\left(\bg^{k-1}, T(\bg^{k-1})+\binom{\varphi_k}{\psi_k}\right), 
\end{equation}
where $U=\binom{\vvv}{\www}\in\MH_N$ is an initial point and 
$$
\varphi_k={\mathsf P}_N\eta_k, \quad \psi_k=(I-{\mathsf P}_N)\eta_k,
\quad T(\UUU)=\binom{T_0(\UUU)}{0}. 
$$
We shall sometimes write $\bg^k(\UUU)$ to indicate the dependence of the random trajectory on the initial point. 
Let us note that, if Condition~(D) is satisfied with $b_j>0$ for all $j\ge1$, then for any function $\fff\in C_b(\MH_N)$ that can be written in the form
\begin{equation} \label{1.12}
\fff(\UUU)=F(v_{1-k},\dots,v_0), \quad \UUU=\bigl(\sigma_j=\tbinom{v_j}{w_j},j\in\Z_0\bigr),
\end{equation}
where $F:(H_N)^k\to\R$ is a measurable function, we have
\begin{equation} \label{1.13}
\E_\UUU \fff(\bg^k)
=\int_{(H_N)^k}D_k(\UUU,\sigma_{1},\dots,\sigma_k)F(v_1,\dots,v_k)\,
\dd\ell_N(\sigma_1)\dots\dd\ell_N(\sigma_k). 
\end{equation}
Here the subscript~$\UUU$ means that we consider the trajectory starting from~$\UUU$, $\ell_N\in\PP(H_N\times H_N^\bot)$ denotes the direct product of the Lebesgue measure~$\LL_N$ on~$H_N$ and the law of~$\psi_1$ (which is a measure on~$H_N^\bot$), $\sigma_j=\tbinom{v_j}{w_j}$ for $j=1,\dots, k$, and 
\begin{equation} \label{1.14}
D_k(\boldsymbol U,\sigma_1,\ldots,\sigma_k)=\prod_{n=1}^k D(v_n-T_0(\boldsymbol U,\sigma_1,\ldots,\sigma_{n-1})),
\end{equation}
where $D(v)$ denotes the density of the law for~$\varphi_1$ with respect to~${\cal L}_N$:  
\begin{equation} \label{1.15}
D(v)=\prod_{i=1}^N b^{-1}_i p_i(b^{-1}_ix_i), 
\quad v=(x_1,\ldots,x_N)\in H_N.
\end{equation}

We now prove a refined version of the Lyapunov--Schmidt type reduction established in Section~3 of~\cite{KS-cmp2000}. 
Given an integer~$N\ge1$ and positive numbers~$R$ and~$b$, define
$$
B_{R,b}:=B_{H_N} (R) \times B_{H _N^\bot}(b), \quad 
\boldsymbol B_{R,b}:=(B_{R,b})^{\Z_0}.
$$
If $\WW_0:\boldsymbol B_{R,b}\to H_N^\bot$ is a continuous mapping, then we extend~$\WW_0$ to the entire space~$\MH_N$ by setting it to zero outside~$\boldsymbol B_{R,b}$. Given such a mapping, we define $T_0:\MH_N\to H_N$ by the relation
\begin{equation} \label{1.19}
T_0(\boldsymbol v, \boldsymbol \psi)
={\mathsf P}_NS(v_0+\WW_0(\boldsymbol v,\boldsymbol\psi)).
\end{equation}

\begin{proposition} \label{p1.2}
Let Conditions {\rm(A)--(D)} be satisfied and let $b=\sqrt\BBBB$. Then for any $R>0$ and $\varkappa>0$ there is an integer $N_*\ge1$ such that for any $N\ge N_*$ one can find a constant $C>0$ and a continuous mapping 
$$
\WW_0:\boldsymbol B_{R,b}\to H_N^\bot, \quad 
\UUU=(\vvv,\ppsi)\mapsto w_0,
$$ 
possessing the following properties.

\medskip
{\bf Lipschitz continuity.} 
For any $\UUU_i=(\vvv_i,\ppsi_i)\in\BBB_{R,b}$, $i=1,2$, we have
\begin{equation} \label{1.16}
\|\WW_0(\UUU_1)-\WW_0(\UUU_2)\|
\le C\sup_{j\le 0}
\bigl\{e^{\varkappa j}\bigl(\|v_{1j}-v_{2j}\|+\|\psi_{1j}-\psi_{2j}\|\bigr)\bigr\},
\end{equation}
where $v_{1j}$ stands for the $j^{\text{th}}$ element of~$\vvv_1$, and we used similar notation for the other sequences. 

\smallskip
{\bf Regularity.}
For any $j\le 0$, the mapping $\WW_0(\vvv,\ppsi)$ is continuously differentiable with respect to~$\varUpsilon_j=\binom{v_j}{\psi_j}$ in the closed ball~$B_{R,b}$.

\smallskip
{\bf Reduction.} 
Let $r\ge0$ be such that $\aA(B_H(r))\subset B_H(R)$ and let $u\in\aA(B_H(r))$ be a vector written as $u=v_0+\WW_0(\UUU)$ for some $\UUU=(\vvv,\ppsi)\in\BBB_{R,b}$. Then the trajectory~$\bg^k=\bg^k(\UUU)$ defined by~\eqref{1.11}, \eqref{1.19} is such that 
\begin{equation} \label{1.17}
\IP\bigl\{\bg^k(\UUU)\in\BBB_{R,b}\mbox{ for all $k\ge0$}\bigr\}=1. 
\end{equation}
Moreover, if a bounded measurable function $f:H^k\to\R$ is of the form 
$$
f(u_1,\dots,u_k)=F({\mathsf P}_Nu_1,\dots,{\mathsf P}_Nu_k),
\quad u_1,\dots,u_k\in H,
$$ 
where $F:(H_N)^k\to\R$ is a measurable function, then
\begin{equation} \label{1.18}
\E_u f(u_1,\dots,u_k)=\E_\UUU\fff(\bg^k)\quad\mbox{for $k\ge0$},
\end{equation}
where $\fff\in C_b(\MH_N)$ is defined as 
$\fff(\vvv,\ppsi)=F(v_{1-m},\dots,v_0)$, and the subscripts~$u$ and~$\UUU$ in~\eqref{1.18} mean that we consider trajectories of~\eqref{1.1} and~\eqref{1.11} starting from~$u$ and~$\UUU$, respectively. 
\end{proposition}

Let us note that the regularity of~$\WW_0$ with respect to~$(v_j,\psi_j)$ and inequality~\eqref{1.16} imply that
\begin{align}\label{E:2.6}
\left\|\frac{\p \WW_0}{\p \varUpsilon_{j}} (\bg)\right\|_{L(H,H _N^\bot)}
\le C e^{\varkappa j},  
\end{align} 
for all $\bg=\binom{\vvv}{\ppsi}\in \BBB_{R,b}$ and $j\le0$. 
 
\begin{proof}[Scheme of the proof of Proposition~\ref{p1.2}]
Let us consider the equation
\begin{align}\label{E:2.4}
\tilde v_k={\mathsf Q}_NS(v_{k-1}+\tilde v_{k-1} )+\psi_k, \quad  k\le 0,
\end{align}
where  ${\mathsf Q}_N=I-{\mathsf P}_N$, and $\{(v_j,\psi_j),j\le0\}\in\BBB_{R,b}$ is a given vector. 
It follows from Theorem~3.1 of~\cite{KS-cmp2000} that if~$N\ge1$ is sufficiently large, then~\eqref{E:2.4} has a unique solution $\boldsymbol{\tilde v}\in (B_{H_N^\bot}(\widetilde R))^{\Z_0}$, where $\widetilde R$ depends on~$R$. We denote by~$\WW_0$ the mapping that takes $(\vvv,\ppsi)$ to the zeroth component of~$\boldsymbol{\tilde v}$. Then~\eqref{1.16} follows immediately from inequality~(3.7) of~\cite{KS-cmp2000}, while relations~\eqref{1.17} and~\eqref{1.18} are consequences of the construction. Thus, it remains to prove the regularity of~$\WW_0$ with respect to~$(v_j,\psi_j)$. This property would follow immediately from the implicit function theorem if~\eqref{E:2.4} was a uniquely solvable equation in a Banach space. Since this is not the case, we apply the following simple argument. 

Let $\chi:H\to\R$ be a smooth cut-off  function equal to~$1$ on the ball $B_H(R+\widetilde R)$. Choosing~$N$ sufficiently large, we derive from~\eqref{1.5} that
\begin{equation} \label{1.22}
\|(I-{\mathsf P}_N)((\chi S)(u_1)-(\chi S)(u_2))\|\le \frac12\|u_1-u_2\|\quad
\mbox{for $u_1,u_2\in H$}. 
\end{equation}
It follows that, for any $(\vvv,\ppsi)\in L^\infty(\Z_0,H_N\times H_N^\bot)$, the equation
\begin{equation}\label{1.23}
\tilde v_k={\mathsf Q}_N(\chi S)(v_{k-1}+\tilde v_{k-1} )+\psi_k, \quad  k\le 0,
\end{equation}
has a unique solution $\boldsymbol{\tilde v}\in L^\infty(\Z_0,H_N^\bot)$. This solution coincides with that of~\eqref{E:2.4} for $(\vvv,\ppsi)\in\BBB_{R,b}$. Since the mapping entering the right-hand side of~\eqref{1.23} is $C^1$-smooth in appropriate spaces, the unique solution~$\boldsymbol{\tilde v}$ of Eq.~\eqref{1.23} will be a $C^1$-smooth function of~$(\vvv,\ppsi)$ if we show that Eq.~\eqref{1.23} can be solved locally with the help of the implicit function theorem. In particular, the dependence of the zeroth component of~$\boldsymbol{\tilde v}$ on~$(v_j,\psi_j)$ will be~$C^1$ for any $j\le0$. 

To prove the applicability of the implicit function theorem, let us define a mapping $G:L^\infty(\Z_0,H_N\times H_N^\bot)\to L^\infty(\Z_0,H_N^\bot)$ by the formula
$$
G(\vvv,\boldsymbol{\tilde v})=\bigl(\tilde v_k-{\mathsf Q}_N(\chi S)(v_{k-1}+\tilde v_{k-1} )-\psi_k, k\in \Z_0\bigr),
$$
where $\vvv=(v_k,k\in\Z_0)$ and $\boldsymbol{\tilde v}=(\tilde v_k,k\in\Z_0)$. Equation~\eqref{1.23} is satisfied if and only if $G(\vvv,\boldsymbol{\tilde v})=0$. The derivative of~$G$ with respect to~$\boldsymbol{\tilde v}$ has the form
$$
G'(\vvv,\boldsymbol{\tilde v})=I-G_1(\vvv,\boldsymbol{\tilde v}),
$$
where $I$ denotes the identity operator in $L^\infty(\Z_0,H_N^\bot)$ and $G_1(\vvv,\boldsymbol{\tilde v})$ is a linear operator in the same space whose norm does not exceed $1/2$ in view of~\eqref{1.22}. Thus, one can apply the implicit function theorem. This completes the proof of the proposition. 
\end{proof}

\subsection{An auxiliary family of Markov chains}
\label{s1.3}
In what follows, we shall need an auxiliary inhomogeneous Markov family in the space~$\MH_N$. Namely, let us fix an integer $r\ge1$, an initial point $\UUU\in\MH_N$,  and a measurable mapping $T_0:\MH_N\to H_N$ and define a random sequence $\{\bg_r^k=\bg_r^k(\UUU),k\ge0\}$ by the relations
\begin{equation} \label{1.24}
 \bg_r^0=\boldsymbol U, \quad 
 \bg_r^k=\left(\bg_r^{k-1}, T(\bg_r^{k-1})+\binom{\varphi_{k,r}}{\psi_k}\right), 
\end{equation}
where $T$ is the same as in~\eqref{1.11}, $\varphi_{k,r}=\varphi_k$ for $k\ne r$, and~$\varphi_{r,r}$ is a random variable in~$H_N$ that is independent of~$\{\eta_k, k\ge1\}$ and is uniformly distributed on the support of the law for~$\varphi_1={\mathsf P}_N\eta_1$. The latter means that
$$
\DD(\varphi_{r,r})=C\,I_{{\mathsf P}_N\KK}(x)\LL_N(dx), \quad x\in H_N,
$$
where $C>0$ is a normalising constant and~$\KK=\supp\DD(\eta_1)$. Note that $\{\bg_r^k(\UUU),\UUU\in\MH_N,k\ge0\}$ is  a family of  inhomogeneous Markov chains in~$\MH_N$ 
satisfying~$\bg_r^k(\UUU)=\bg^k(\UUU)$ for $k\le r-1$ and for $k\ge r+1$. In particular, the Markov property implies that, if $g:H^m\to\R$ is a bounded measurable function, then
\begin{equation} \label{1.25}
\E\bigl(g\bigl(\bg_r^{r+1}(\UUU),\dots,\bg_r^{r+m}(\UUU)\bigr)\,|\,\FF_r\bigr)
=\E\,g\bigl(\bg^1(\VVV),\dots,\bg^m(\VVV)\bigr)\big|_{\VVV=\bg_r^r(\UUU)}\,,
\end{equation}
where ${\cal F}_r$ is is the $\sigma$-algebra generated by $\{ \bg^{j}\}_{j\leq r}$.

Let~$\MU\subset\MH_N$ be the domain of attainability from zero for~$\{\bg^k\}$. That is, we define
$$
\MU_0=\{\mathbf0\}, \qquad 
\MU_k=\bigl\{(\bg,T(\bg)+\KK),\bg\in\MU_{k-1}\bigr\}\quad\mbox{for $k\ge1$},
$$
where $\mathbf0\in\MH_N$ stands for the zero element and~$\KK$ is the support of~$\DD(\eta_1)$, and denote by~$\MU$ the closure of the union of~$\MU_k$, $k\ge0$, in the Tikhonov topology of~$\MH_N$. Since $\supp\DD(\varphi_{k,r})=\supp\DD(\varphi_k)$ and, hence, the support of the law for~$\binom{\varphi_{k,r}}{\psi_k}$ is equal to~$\KK$, the domain of attainability from zero for~$\{\bg_{r}^k\}$ coincides with~$\MU$. In particular, $\MU$~is an invariant subset for~$\{\bg_r^k\}$. 

We shall also need the following property of~$\MU$ established in Section~3.2 of~\cite{KS-cmp2000}. Let $R_*>0$ be such that $b=\sqrt\BBBB\le R_*$ and $\aA\subset B_H(R_*)$. Then
\begin{equation} \label{1.26}
\MU\subset \bigl(B_H(R_*)\bigr)^{\Z_0}.
\end{equation}

\subsection{Proof of Theorem~\ref{main-tech}}
\label{s4.3}
{\it Step 1}. 
We first reduce the proof of the uniform Feller property to a similar question for the Markov family~$\{\bg^k\}$ with the state space~$\MH_N$ (see Section~\ref{s1.2}). Recall that the relationship between~$\{\bg^k\}$ and~$\{u_k\}$  is described in Proposition~\ref{p1.2} and the domain of attainability from zero~$\MU$ was defined in Section~\ref{s1.3}. Given a function $V\in C_b(H_N)$, we introduce a semigroup acting on~$C(\MU)$ by the formula
$$
(\bPPPP_k^V\fff)(\bg)=\E_\bg\exp\biggl(\,\sum_{n=1}^kV\bigl({\mathsf P}_N\varUpsilon_{1-n}^k)\bigr)\biggr)\fff(\bg^k) ,
\quad k\ge1,
$$
where $\bg^k=(\varUpsilon_j^k, j\in\Z_0)$ and $\fff\in C(\MU)$. To any function $f\in C(H_N)$ we associate $\fff\in C(\MU)$ defined by 
\begin{equation} \label{4.8}
\fff(\bg)=f(v_0), \quad 
\bg=\tbinom{\vvv}{\ppsi}=\bigl(\tbinom{v_j}{\psi_j},j\in\Z_0\bigr).
\end{equation}
We claim that the uniform Feller property for $\{\PPPP_k^V\}$ with the determining class~$\CC$ of  Theorem \ref{main-tech} will be established if we prove the following assertion:

\begin{description}
\item[(P)] 
{\sl Let $R_*>0$ be the constant for which~\eqref{1.26} holds, let $\varkappa>0$ satisfy the inequality
\begin{equation} \label{4.9}
\varkappa>\Osc_\aA(V):=\sup_\aA V-\inf_\aA V, 
\end{equation}
and let $N_*=N_*(R_*,\varkappa)\ge1$ be the integer constructed in Proposition~\ref{p1.2}. Then, for any integer $N\ge N_*$, any function $V\in C(\aA)$ representable in the form~\eqref{4.4}, and any function $\fff\in C(\MU)$ of the form~\eqref{4.8} with $f\in C_+(H_N)$, the sequence $\{\|\bPPPP_k^V\fff\|_\infty^{-1}\bPPPP_k^V\fff,k\ge1\}$ is relatively compact in~$C(\MU)$.}
\end{description}
Indeed, let $V\in\VV$,  $f\in \CC$, and let~$\fff$ be the function defined by~\eqref{4.8}. There is no loss of generality in assuming that $N\ge1$ is so large that the conclusion of~(P) holds. Thus, any sequence of  positive integers going to~$+\infty$ contains a subsequence~$\{k_n\}$ such that
\begin{equation} \label{4.10}
\sup_{\bg\in\MU}\,\bigl|\|\bPPPP_{k_m}^V\fff\|_\infty^{-1}\bPPPP_{k_m}^V\fff(\bg)-\|\bPPPP_{k_n}^V\fff\|_\infty^{-1}\bPPPP_{k_n}^V\fff(\bg)\bigr|\to0
\quad\mbox{as $m,n\to\infty$}. 
\end{equation}
On the other hand, if $u\in \aA$ is such that $u=v_0+\WW_0(\bg)$ for some $\bg\in\MU$, then by~\eqref{1.18} we have
\begin{align}
(\bPPPP_k^V\fff)(\bg)
&=\E_\bg\exp\biggl(\,\sum_{n=1}^kV({\mathsf P}_N\varUpsilon_{1-n}^k)\biggr) f(\varUpsilon_0^k) \notag\\
&=\E_u \exp\biggl(\,\sum_{n=1}^kV({\mathsf P}_Nu_n)\biggr) f({\mathsf P}_Nu_k)  =(\PPPP_k^Vf)(u). \label{4.11}
\end{align}
Since the image of~$\MU$ by the mapping $\bg=\bigl(\binom{v_j}{\psi_j},j\in\Z_0\bigr)\mapsto v_0+\WW_0(\bg)$ coincides with~$\aA$ (see Section~3 in~\cite{KS-cmp2000}), it follows from~\eqref{4.10} and~\eqref{4.11} that 
$$
\sup_{u\in\aA}\,\bigl|\|\PPPP_{k_m}^Vf\|_\infty^{-1}\PPPP_{k_m}^Vf(u)-\|\PPPP_{k_n}^Vf\|_\infty^{-1}\PPPP_{k_n}^Vf(u)\bigr|\to0
\quad\mbox{as $m,n\to\infty$}. 
$$
By the  Arzel\`a theorem, what has been established implies that the sequence $\{\|\PPPP_{k}^Vf\|_\infty^{-1}\PPPP_{k}^Vf\}$ is uniformly equicontinuous. Since $V\in\VV$ and $f\in\CC$ were arbitrary, we obtain the required uniform Feller property. 

\smallskip
{\it Step 2}. 
We now prove Property~(P). To this end, it suffices to find positive constants~$C$ and~$c$ such that
\begin{align}\label{E:3.9}
\|\bPPPP_k^V\fff\|_\ty^{-1}  \sup_{\bg\in \MU}\,
\biggl\|\frac{\p \bPPPP_k^V \fff}{\p \varUpsilon_j}(\bg)\biggr\|
\le C e^{cj},
\end{align}
where $\bg=(\varUpsilon_j,j\in\Z_0)$. 
Indeed, let us endow the set~$\MU$ with the metric
$$
d_c(\bg_1,\bg_2)=\sum_{j\in\Z_0}e^{cj}\|\varUpsilon_j^1-\varUpsilon_j^2\|,
\quad \bg_i=\bigl(\varUpsilon_j^i,j\in\Z_0\bigr).
$$
The topology defined by~$d_c$ on~$\MU$ coincides with the Tikhonov topology. Let us set $\bgg_k=\|\bPPPP_k^V\fff\|_\ty^{-1}\bPPPP_k^V \fff$. If we show that~$\bgg_k$ is $C$-Lipschitz continuous on~$\MU$ with respect to the metric~$d_c$ for any $k\ge1$, then the relative compactness of~$\{\bgg_k\}$ in~$C(\MU)$ will follow by the Arzel\`a theorem. 

Assuming that~\eqref{E:3.9} is established, we use the mean value theorem with respect to the variable~$\varUpsilon_j=\binom{v_j}{\psi_j}$ to estimate the Lipschitz constant of~$\bgg_k$:
\begin{align*}
|\bgg_k(\bg_1)-\bgg_k(\bg_2)|
&\le \sum_{j\in\Z_0}\Lip_{j}(\bgg_k)\|\varUpsilon_j^1-\varUpsilon_j^2\|\\
&\le \sum_{j\in\Z_0}Ce^{ cj}\|\varUpsilon_j^1-\varUpsilon_j^2\|
=C\,d_c(\bg_1,\bg_2),
\end{align*}
where $\Lip_j(\bgg_k)$ stands for the Lipschitz constant of~$\bgg_k$ with respect to~$\binom{v_j}{\psi_j}$. This inequality shows that~$\bgg_k$ is $C$-Lipschitz continuous.

\smallskip
{\it Step 3}. Let us prove \eqref{E:3.9}. In view of~\eqref{1.13}, we have
\begin{align}\label{4.12*}
(\bPPPP_k^V\fff)(\bg)=\int
D_k(\bg,\sigma_{1},\dots,\sigma_k)\exp\biggl(\,\sum_{n=1}^kV(v_n)\biggr)f(v_k)\,\dd\ell_N(\sigma_1)\dots\dd\ell_N(\sigma_k).
\end{align}
Here and henceforth, the integrals without limits are taken over $(B_{R,b})^k$ and, as before, we write $\sigma_n=\tbinom{v_n}{w_n}$ for $n=1,\dots, k$. Taking the derivative of the above relation with respect to~$\varUpsilon_{j}$ and defining~$\eell_N^k$ as the direct product of~$k$ copies of~$\ell_N$, we get 
\begin{equation} \label{4.13}
\frac{\bPPPP_k^V\fff}{\p\varUpsilon_j}(\bg)
=\int
\frac{D_k(\bg,\ssigma_k)}{\p\varUpsilon_j}
\exp\biggl(\,\sum_{n=1}^kV(v_n)\biggr)f(v_k)\,
\dd\eell_N^k(\ssigma_k),
\end{equation}
where we set $\ssigma_n=(\sigma_{1},\dots,\sigma_n)$. Note that
$$
\frac{\p D_k(\bg,\ssigma_k)}{\p\varUpsilon_j}
=\sum_{r=1}^kD_{kr}(\bg,\ssigma_r)
\biggl\langle\frac{\p D}{\p v}\bigl(v_r-T_0(\bg,\ssigma_{r-1})\bigr),
\frac{\p T_0}{\p \varUpsilon_j}(\bg,\ssigma_{r-1})\biggr\rangle,
$$
where
$$
D_{kr}(\bg,\ssigma_k)=\prod_{r\neq n=1}^k 
D\bigl(v_n-T_0(\bg,\ssigma_{n-1})\bigr).
$$ 
In view of~\eqref{1.16}, \eqref{E:2.6}, and the regularity of~$S$, we have
$$
\biggl\|\frac{\p T_0}{\p \varUpsilon_{ j}} (\bg,\ssigma_{r-1})\biggr\|_{L(H,H_N)}
\le C_1e^{ \varkappa(j-r)} \quad 
\text{for any $(\bg,\ssigma_{r-1})\in{\boldsymbol B}_{R,b}$},
$$
where we denote by $C_i$ unessential positive constants. 
Substituting the above relations into~\eqref{4.13}, we derive 
\begin{equation} \label{4.14} 
\biggl\|\frac{\bPPPP_k^V\fff}{\p\varUpsilon_j}(\bg)\biggr\|
\le C_2(N)\sum_{r=1}^ke^{ \varkappa(j-r)}
\int D_{kr}(\bg,\ssigma_k)\exp\biggl(\,\sum_{n=1}^kV(v_n)\biggr)f(v_k)\,
\dd\eell_N^k(\ssigma_k).
\end{equation}
Let us recall that, given an integer $N\ge1$, the inhomogeneous family of Markov chains~$\{\bg_r^k=(\varUpsilon_{r,j}^k, j\in\Z_0)\}$ was defined in Section~\ref{s1.3}. The integral on the right-hand side of~\eqref{4.14} can be rewritten as
$$
I_{kr}(\bg):=\frac{1}{\LL_N({\mathsf P}_N\KK)}\,\E_\bg\biggl\{\exp\biggl(\,\sum_{n=1}^k
V\bigl({\mathsf P}_N\varUpsilon_{r,n}^k\bigr)\biggr)
f\bigl({\mathsf P}_N\varUpsilon_{r,k}^k\bigr)\biggr\}.
$$
Conditioning on~$\FF_r$ and using the Markov property~\eqref{1.25}, we can estimate $I_{kr}$ as follows:
\begin{align*}
|I_{kr}(\bg)|&\le\frac{\|f\|_\infty\exp(r\sup_\aA V)}{\LL_N({\mathsf P}_N\KK)}
\E_\bg\E_\bg\biggl\{\exp\biggl(\,\sum_{n=r+1}^k
V\bigl({\mathsf P}_N\varUpsilon_{r,n}^k\bigr)\biggr)\,\bigg|\,\FF_r\biggr\}\\
&=C_3\exp(r\sup\nolimits_\aA V)\,\E_\bg\E_{\bg_r^r}
\biggl\{\exp\biggl(\,\sum_{n=1}^{k-r}
V\bigl({\mathsf P}_N\varUpsilon_{n}^k\bigr)\biggr)\biggr\}\\
&\le C_3\exp(r\sup\nolimits_\aA V)\,
\sup_{\UUU\in\MU}\E_\UUU\biggl\{\exp\biggl(\,\sum_{n=1}^{k-r}
V\bigl({\mathsf P}_N\varUpsilon_{n}^k\bigr)\biggr)\biggr\}\\
&\le C_3\exp\bigl(r\Osc_\aA(V)\bigr)\,\|\bPPPP_k^V{\mathbf1}\|_\infty.
\end{align*}
Substituting this estimate into~\eqref{4.14}, we obtain
\begin{equation*} 
\biggl\|\frac{\bPPPP_k^V\fff}{\p\varUpsilon_j}(\bg)\biggr\|
\le C_4(N)\,e^{\varkappa j}
\|\bPPPP_k^V{\mathbf1}\|_\infty
\sum_{r=1}^ke^{-r(\varkappa-\Osc_\aA(V))}.
\end{equation*}
Recalling~\eqref{4.9}, we see that
\begin{equation} \label{4.15}
\sup_{\bg\in\MU}\,\biggl\|\frac{\bPPPP_k^V\fff}{\p\varUpsilon_j}(\bg)\biggr\|
\le C_5(N)e^{\varkappa j}\|\bPPPP_k^V{\mathbf1}\|_\infty.
\end{equation}
On the other hand, since $\fff$ is continuous and positive on the compact set~$\MU$, it can be minorised by a constant $c_\fff>0$. It follows that
$$
\bPPPP_k^V\fff(\bg)\ge c_\fff(\bPPPP_k^V{\mathsf1})(\bg),
$$
and  we conclude that 
$$
\|\bPPPP_k^V\fff\|_\infty\ge c_\fff\|\bPPPP_k^V{\mathsf1}\|_\infty. 
$$
Comparing this with~\eqref{4.15}, we obtain inequality~\eqref{E:3.9} with $C=C_5(N)/c_\fff$ and $c=\varkappa$. The proof of Theorem 
\ref{main-tech} is complete.

\section{Proof of  Theorem \ref{main}}
\label{s4}
We shall prove Theorem \ref{main} by verifying Property 1 (the existence of a limit) and Property 2 (uniqueness of the equilibrium state) 
of Section \ref{s2.3}. 

Let $P_k^V(u, \Gamma)$, $\{\PPPP _k^V\}$, $\VV$, and ${\cal C}$  be as in Theorem \ref{main-tech}. 
For any $V\in C(\aA)$,
$$
P_k^V(u,\cdot)\ge e^{-k\|V\|_\infty}P_k(u,\cdot) 
\quad\mbox{for any $u\in\aA$}.
$$
Since $P_k(u,\Gamma)$ is uniformly irreducible (e.g., see Section~5 of~\cite{KS-cmp2000} for a proof of a similar assertion in a more complicated setting), so is $P_k^V(u,\Gamma)$. By Theorem~\ref{main-tech}, for any $V\in {\cal V}$  the semigroup~$\{\PPPP _k^V\}$ possesses the uniform Feller property for  the determining class~${\cal C}$. Thus, for $V\in {\cal V}$, Theorem~\ref{t3.1} holds for 
the semigroup $\{\PPPP _k^V\}$ and the class ${\cal C}$. 

We  now turn to the proof of Property~1 and the existence of the limit~\eqref{E:3.1*}.
Theorem~\ref{t3.1} implies that for  any $V\in {\cal V}$ there is  $h_V\in C_+(\aA)$  and a constant $\la_V>0$ such that 
$$
\la_V^{-k}\PPPP _k^{V} \mathbf1 \to h_V 
\quad\text{in $C(\aA)$ as $k\to\infty$}.
$$
It follows that for  $V\in {\cal V}$
\begin{align}\label{E:3.4}
Q(V)=\lim_{k\to+\ty} \frac{1}{k}\log (\PPPP _k^V \mathbf1)(u)=\log \lambda_V
\end{align}  
uniformly in~$u\in  \aA$. The estimate 
\begin{align*}
(\PPPP _k^{V_1} \mathbf1)(u)
&=\E_u\exp\biggl(\,\sum_{n=1}^kV_1(u_n)\biggr)
\le e^{k\|V_1-V_2\|_\infty}\,\E_u\exp\biggl(\,\sum_{n=1}^kV_2(u_n)\biggr)\\
&=e^{k\|V_1-V_2\|_\infty}\,(\PPPP _k^{V_2} \mathbf1)(u),
\end{align*}
which holds for  any $V_1, V_2\in C(\aA)$, implies
\begin{align}\label{E:3.6}
\biggl| \frac{1}{k}\log (\PPPP_k^{V_1}\mathbf)(u)
-\frac{1}{k}\log (\PPPP_k^{V_2} \mathbf1)(u)\biggr|
\le \|V_1 -V_2 \|_\ty\quad\text{for $k\ge1$, $u\in \aA$}. 
\end{align}
Hence,  (\ref{E:3.4}) holds for all $V\in C(\aA)$, the limit is uniform in $u \in \aA$, and 
\begin{align}\label{E:3.6*}
|Q(V_1)-Q(V_2)| \le\|V_1 -V_2\|_\ty\quad
\mbox{for $V_1,V_2\in C(\aA)$}. 
\end{align}
The existence of the  limit~\eqref{E:3.1*} for an arbitrary $\aA$-valued initial function~$u_0$ now follows  by integration with respect to the law of~$u_0$.  

\medskip
We now turn to Property~2. We shall show 
that  for any $V\in \VV$, there is only one equilibrium state $\sigma_V\in \PP(\aA)$ for $Q(V)$. We begin with three auxiliary lemmas. 

\smallskip
The functional $Q :C(\aA)\to\R$ is $1$-Lipschitz continuous and convex. Let $I: \MM(\aA)\to \R$ be its Legendre transform. A proof of the following result can be found in Section~6.5.1 of~\cite{ADOZ00}.

\begin{lemma}\label{L:DV}
For any $\sigma\in\PP(\aA)$, we have
\begin{align}\label{E:infb}
I(\sigma)=-\inf_{g\in C_+(\aA)} \biggl\langle\log \biggl(\frac{\PPPP_1g}{g}\biggr),\sigma\biggr\rangle.
\end{align} 
\end{lemma}

Until the end of the proof, we fix $V\in \VV$. We denote by 
$$
(\la_V, \mu_V, h_V)\in \R_+\times \PP(\aA)\times C_+(\aA)
$$ 
the  triple constructed in Theorem~\ref{t3.1} for the generalised Markov semigroup~$\{\PPPP_k^V\}$.  Recall that the Markov semigroup~$\{\SSS^V_k\}$ is defined by~\eqref{2.22} and denote by~$\{\SSS^{V*}_k\}$ its dual semigroup acting on~$\PP(\aA)$. The following lemma establishes the uniqueness of stationary measure for~$\{\SSS^{V*}_k\}$ and provides a formula for it. 

\begin{lemma}\label{L:une}  $\nu_V=h_V\mu_V$ is the unique stationary measure for~$\{\SSS^{V*}_k\}$.
\end{lemma}

\begin{proof} The relation  $\lag h_V,\mu_V\rag =~1$ implies that $\nu_V \in {\cal P}({\cal A})$.
For any $g\in C(\aA)$ we have
$$
\lag \SSS^V_1 g, \nu_V \rag
= \la_V^{-1}\lag    \PPPP_1^V (gh_V), \mu_V \rag
=  \lag     gh_V , \mu_V \rag
=\lag   g , \nu_V \rag,
$$
and so  $\nu_V$ is a  stationary measure for~$\SSS^{V*}_1$. 
Let~$\nu\in \PP(\aA)$ be another stationary measure for~$\SSS^{V*}_1$. Then the measure~$\mu$ defined by $\mu=h_V^{-1}\nu$ satisfies the relation
$$
\lag      gh_V , \mu \rag= \lag      g , \nu \rag= \lag \SSS_1^V g, \nu \rag=\la_V^{-1} \lag \PPPP_1^V  ( gh_V),  \mu \rag
\quad\mbox{for any $g\in C(\aA)$}.
$$
Since~$h_V $ is everywhere positive, this relation implies $\PPPP_1^{V*}  \mu=\la_V\mu$ and  so  $\mu=C\mu_V$. Since $\langle h_V,\mu\rangle=\langle h_V,\mu_V\rangle=1$, we conclude that $\mu=\mu_V$ and $\nu=\nu_V$. 
\end{proof}

\begin{lemma}\label{L:kn}
The infimum in~\eqref{E:infb} is attained at the function $g=h_Ve^V$ if and only if $\sigma=\nu_V$. 
\end{lemma}
 
\begin{proof}
Let us define $\tilde h_V := h_V  e^V $, so that $e^V \PPPP_1 \tilde h_V=\la_V \tilde h_V$. Assume that   $I(\sigma)=-  \bigl\lag \log \bigl(\frac{\PPPP_1 \tilde h_V}{\tilde h_V}\bigr), \sigma\bigr\rag$ for some $\sigma\in \PP(\aA)$. Then, for any $g\in C_+(\aA)$, we have 
$$
\frac{\dd}{\dd \e} \biggl\lag\log\biggl(\frac{\PPPP_1(\tilde h_V+\e g) }{\tilde h_V+\e g}\biggr),\sigma \biggr\rag\bigg|_{\e=0}=0.
$$This implies that
$$
  \biggl\lag \biggl(\frac{ \PPPP_1 g }{\PPPP_1 \tilde h_V }-\frac{  g }{\tilde h_V }\biggr),\sigma \biggr\rag =0,
$$
which is equivalent to the relation $\lag \SSS^V_1  g,\sigma\rag=\lag   g,\sigma\rag$. Since this is true for any $g\in C_+(\aA)$,  Lemma~\ref{L:une} implies  that $\sigma= \nu_V$.

Now let us prove that  $I(\nu_V)=-  \bigl\lag \log \bigl( \frac{\PPPP_1 \tilde h_V}{\tilde h_V}\bigr), \nu_V\bigr\rag$. We need to show that 
$$
\biggl\lag \log \biggl( \frac{\PPPP_1 \tilde h_V}{\tilde h_V}\biggr), \nu_V\biggr\rag
\le \biggl\lag \log \biggl( \frac{\PPPP_1 g}{g}\biggr), \nu_V\biggr\rag
\quad\mbox{for any $g\in C_+(\aA)$}. 
$$
The relation is equivalent to
\begin{align}\label{E:bel}
\biggl\lag \log \biggl( \frac{\tilde h_V \PPPP_1 g}{g\,\PPPP_1 \tilde h_V }\biggr), \nu_V\biggr\rag\le0.
\end{align}
Since $\nu_V$ is a stationary measure  for~$\SSS^{V*}_1$, we have $\lag \SSS^V_1  f-f,\nu_V\rag=0$ for any $f\in C(\aA)$. Combining this with Jensen's inequality, we see that $$\lag \log(\SSS^V_1  e^f)-f,\nu\rag\ge 0.$$ Taking $f=\log g$ in this equality and using the definition of $\SSS^V_1 $, we get  
$$
\biggl\lag \log \biggl( \frac{ \PPPP_1 ( \tilde h_V g)}{g\,\PPPP_1  \tilde h_V }\biggr), \nu_V\biggr\rag\ge0.
$$
Replacing $g$ by $ {g}/{\tilde h_V}$, we obtain~(\ref{E:bel}).
\end{proof}

We can now establish the uniqueness of equilibrium state. Let~$\sigma_V\in\PP(\aA)$ be such that~\eqref{E:3.1**} holds. The definition of~$\tilde h_V$ and the relation $Q(V)=\log\lambda_V$ imply that
\begin{align*}
\lag V,\sigma_V\rag-I(\sigma_V)=Q(V)
=\biggl\lag V+ \log \biggl( \frac{\PPPP_1 \tilde h_V}{\tilde h_V}\biggr),\sigma_V\biggr\rag. 
\end{align*}
It follows that $I(\sigma_V)=-\bigl\lag \log \bigl( \frac{\PPPP_1 \tilde h_V}{\tilde h_V}\bigr)  , \sigma_V\bigr\rag$. By Lemma~\ref{L:kn}, we have $\sigma_V=\nu_V$. This completes the proof of uniqueness of the equilibrium measure for $V\in {\cal V}$ and Theorem \ref{main} follows.

\section{Proof of Theorem~\ref{mainty}} 
\label{s6}
As described in Section \ref{s2.4}, Theorem \ref{mainty} follows from Theorem \ref{mainm} (which in turn is a generalisation 
of Theorem~\ref{main}). To establish Theorem \ref{mainm}, one follows  the general scheme used in the proof of Theorem~\ref{main}, applying it to the Markov chain formed by the segments of trajectories of length~$m$. Namely, let us consider the following family of Markov chains in~$\aA^{(m)}$:
\begin{equation} \label{4.21}
\uuu_k=\Ss(\uuu_{k-1})+\eeta_k,
\end{equation}
where $\uuu_k=(u_k^1, \dots,u_k^{m})$, $\eeta_k=(0,\dots,0,\eta_{k+m-1})$, and $\Ss:H^m\to H^m$ is the mapping given by
$$
\Ss(v_1,\dots,v_m)=\bigl(v_2,\dots,v_m,S(v_m)\bigr), 
\quad (v_1,\dots,v_m)\in\HHH.
$$
It is clear that if~$u_0$ is an $\aA$-valued random variable independent of~$\{\eta_k\}$ and~$\{u_k\}$ is the corresponding trajectory of~\eqref{1.1}, then~$\zeta_k^m$ is the occupation measure for the trajectory of~\eqref{4.21} starting  from the (random) initial point~$(u_0,\dots,u_{m-1})$. Since its law is supported by~$\aA^{(m)}$, the LDP for~$\zeta_k^m$ will be established if we prove the LDP for the Markov family~\eqref{4.21} restricted to the invariant compact set~$\aA^{(m)}$. By Kifer's theorem and the argument described  Section~\ref{s2.3}, the latter result is a consequence of the following two properties (which were described in Section~\ref{s2.3} for~$\{u_k\}$):

\begin{description}
\item[Property 1': The existence of a limit.]
For any function $V\in C(\aA^{(m)})$ the  limit 
\begin{align}\label{4.22}
Q_m(V)=\lim_{k\to+\ty} \frac{1}{k}
\log\E\,\exp\biggl(\,\sum_{n=1}^kV(\uuu_n)\biggr).
\end{align} 
exists and does not depend on the initial point $\uuu=(u^1,\dots,u^m)\in \aA^{(m)}$.
\item[Property 2': Uniqueness of equilibrium state.]
There exists a dense vector space $\VV_m\subset C(\aA^{(m)})$ such that, for any $V\in\VV_m$, there is a unique measure $\sigma_V\in\PP(\aA^{(m)})$ satisfying the relation
$$
Q_m(V)=\sup_{\sigma\in\PP(\aA^{(m)})}
\bigl(\langle V,\sigma\rangle-I_m(\sigma)\bigr),
$$
where $I_m(\sigma)$ denotes the Legendre transform of~$Q_m$.
\end{description}
To establish these assertions, we introduce a generalised Markov semigroup by the relation (cf.~\eqref{2.18})
\begin{equation} \label{4.23}
\PPPP_k^V f(\uuu):= 
\E_{\uuu}f(\uuu_k)\exp \biggl(\,\sum_{n=1}^kV(\uuu_n)\biggr), 
\quad f\in C(\aA^{(m)}),
\end{equation}
where $V\in C(\aA^{(m)})$ is a given function. If we prove that~$\{\PPPP_k^V\}$ satisfies the uniform Feller and uniform irreducibility properties of Theorem~\ref{t3.1} for any~$V$ belonging to a dense subspace~$\VV_m$ that contains constant functions, then the required results will following line by line 
the proof of Theorem \ref{main}. 

\smallskip
To show the uniform irreducibility, note that~$\aA^{(m)}$ is the domain of attainability from zero for system~\eqref{4.21}. Therefore the required 
property follows by repeating the proof of a similar property for~\eqref{1.1}. 

\smallskip
We now turn to the uniform Feller property. Let~$\VV_m$ be the space of functions $V\in C(\aA^{(m)})$ for which there is an integer $N\ge1$ and a function $F\in C(H_N^m)$ such that
$$
V(\uu)=F({\mathsf P}_N\uuu)\quad
\mbox{for $\uuu=(u^1,\ldots,  u^m)\in\aA^{(m)}$},
$$
where ${\mathsf P}_N\uu=({\mathsf P}_Nu^1,\dots,{\mathsf P}_Nu^m)$. Then, for $V,f\in\VV_m$ and $k\ge m$, we can write (cf.~\eqref{4.11} and~\eqref{4.12*})
\begin{multline} \label{4.24}
(\PPPP_k^Vf)(\uuu)=\int_{B_{R,b}^k}
D_k(\bg,\sigma_{1},\dots,\sigma_k)
\exp\biggl(\,\sum_{n=1}^kV(v_{n-m+1},\dots,v_n)\biggr)\\
\times f(v_{k-m+1},\dots,v_k)\,\dd\ell_N(\sigma_1)\dots\dd\ell_N(\sigma_k).
\end{multline}
Here $\bg=\bigl(\tbinom{v_j}{\psi_j},j\in\Z_0\bigr)\in\MU$ is an arbitrary point such that 
$$
u^{l+m}=v_l+\WW_0(\bg_l)\quad\mbox{for $1-m\le l\le0$},
$$
where $\bg_l=\bigl(\tbinom{v_j}{\psi_j},j\in\Z_l\bigr)$, and we use the same notation as in~\eqref{4.12*}. Let us denote by $f_k(\bg)$ the right-hand side of~\eqref{4.24}. As was explained in Section~\ref{s4.3}, the uniform equicontinuity of~$\{\PPPP_k^Vf,k\ge0\}$ follows from the 
uniform equicontinuity of $\{f_k(\bg),k\ge m\}$ (note that these functions act  on~$\MU$). The latter property can be proved by literal repetition of the argument used to establish assertion~(P) of Section~\ref{s4.3}, provided that $V\in\VV_m$ and $f\in \VV_m\cap C_+(\aA^{(m)})$. This completes the proof of Theorems~\ref{mainm} and \ref{mainty}.

\section{Appendix}
\label{s5} 
 In this section, we recall three results on the large deviation principle (LDP). The first of them was established by Kifer~\cite{kifer-1990} and  provides a sufficient condition for the validity of LDP for a family of random probability measures. The second result shows that, when studying the LDP for occupation measures of random processes, one can take the average starting from any non-negative time. The third result due to Dawson and G\"artner~\cite{DG-1987} shows that the process level LDP is a straightforward consequence of the LDP for finite segments of solutions. 

\subsection{Kifer's sufficient condition for LDP}
 \label{s5.1}
Let $\Theta$ be a directed set, let~$X$ be a compact metric space, and let $(\Omega, \FF, \pP)$ be a probability space. We consider a family $\{\zeta_\theta\}=\{\zeta_\theta^\omega\}$ of random probability measures on~$X$ depending on~$\theta\in\Theta$ such that the following  limit exists  for any $V\in C(X)$:
 \begin{align}\label{E:H1}
Q(V)=\lim_{\theta\in \Theta} \frac{1}{r(\theta)}\log \int_{\Omega} \exp\biggl(r(\theta)\int_X V\dd \zeta_\theta^\om\biggr)\dd \pP(\om),
\end{align}   
where $r:\Theta\to\R$ is a given positive function such that $\lim_{\theta\in \Theta} r(\theta)=\ty$. 
Then $Q:C(X)\to \R$ is a convex $1$-Lipschitz functional such that $Q(V)\ge0$ for any $V\in C_+(X)$ and $Q(C)=C$ for any constant $C\in\R$. Recall that the Legendre transform of~$Q$ is defined on the space~$\MM(X)$ by 
\begin{equation} \label{E:H4}
I(\sigma)=\sup_{V\in C(X)}\bigl(\lag V, \sigma\rag-Q(V)\bigr)
\end{equation}
if  $\sigma \in {\cal P}(X)$ and $I(\sigma)=\infty$ otherwise. The function~$I(\sigma)$ is convex and lower semicontinuous in the weak topology, and~$Q$ can be reconstructed by the formula
$$
Q(V)= \sup_{\sigma\in\PP(X)} \bigl(\lag V, \sigma\rag-I(\sigma)\bigr).
$$
Since    $\PP(X)$ endowed with topology of weak convergence  is compact, for any $V\in C(X)$ there is $\sigma_V\in\PP(X)$ such that
\begin{align}\label{E:H3}
Q(V)= \lag V,\sigma_V\rag-I(\sigma_V).
\end{align} 
Any measure $\sigma_V\in\PP(X)$ satisfying~\eqref{E:H3} is called an {\it equilibrium state\/} for~$V$. The following result of Kifer shows that 
if the equilibrium state is unique for a dense vector subspace of $V\in C(X)$, then the LDP holds for~$\zeta_\theta$.

\begin{theorem}\label{T:Kif}
Suppose that limit~\eqref{E:H1} exists for any $V\in C(X)$. Then the LD upper bound
$$
\limsup_{\theta\in \Theta} \frac1{r(\theta)}\log \pP\{\zeta_\theta\in F\}\le -I(F)
$$
 holds with the rate function~$I$ given by~\eqref{E:H4}. Furthermore, if there  exists a dense vector space $\VV\subset C(X)$ such that the equilibrium state $\sigma_V\in \PP(X)$ is unique for any $V\in \VV $, then the LD lower bound also holds:
 $$
 \liminf_{\theta\in \Theta} \frac1{r(\theta)}\log \pP\{\zeta_\theta\in G\}\ge -I(G).
$$
\end{theorem}

\subsection{Exponential equivalence of random probability measures}
\label{s5.3}
Let~$X$ be a Polish space and let~$\{\mu_k\}$ and~$\{\mu_k'\}$ be two sequences of random probability measures on~$X$. Recall that~$\{\mu_k\}$ and~$\{\mu_k'\}$ are said to be {\it exponentially equivalent\/} if 
\begin{equation} \label{5.31}
\lim_{k\to\infty}\IP\bigl\{\|\mu_k-\mu_k'\|_L^*>\delta\bigr\}^{1/k}=0
\quad\mbox{for any $\delta>0$}. 
\end{equation}
It is well known that if two sequences of random probability measures are exponentially equivalent, then an LDP with a good rate function for one of them implies the same LDP for the other; see Section~4.2.2 in~\cite{ADOZ00}. 

\smallskip
Now let $\{u_n\}$ be a random sequence in~$X$. We denote by $\zeta_k^{(m)}$ the occupation measures for~$\{u_n\}$ starting at time~$m\ge0$:
$$
\mu_k^{(m)}=\frac1k\sum_{n=m}^{m+k-1}\delta_{u_n}. 
$$
The following result was used in Sections~\ref{s2.3}.

\begin{lemma} \label{l5.2}
The sequences~$\mu_k^{(m)}$ and~$\mu_k^{(l)}$ are exponentially equivalent for any integers~$m,l\ge0$.
\end{lemma}

\begin{proof} 
Let $f\in L_b(X)$ be  such that $\|f\|_L\le1$. Then 
$$
\bigl|\bigl(f,\mu_k^{(m)}\bigr)-\bigl(f,\mu_k^{(l)}\bigr)\bigr|
\le\frac{2|m-l|}{k}. 
$$
It follows that $\bigl\|\mu_k^{(m)}-\mu_k^{(l)}\bigr\|_L^*\le\frac{2|m-l|}{k}$, whence we see that
$$
\IP\bigl\{\|\mu_k^{(m)}-\mu_k^{(l)}\bigr\|_L^*>\delta\bigr\}=0
\quad\mbox{for $k>2\delta^{-1}|m-l|$}. 
$$
Hence, condition~\eqref{5.31} is satisfied for any $\delta>0$, and the sequences in question are exponentially equivalent. 
\end{proof}

\subsection{Dawson--G\"artner theorem}
\label{s5.2}
For a given   Polish space~$X$, we denote by $\XXX=X^{\Z_+}$ the direct product of countably many copies of~$X$, endowed with the Tikhonov topology, and by $p_m:\XXX\to X^m$ the natural projection to the first~$m$ components of~$\XXX$. Let $\{\zzeta_k\}=\{\zzeta_k^\omega\}$ be a sequence of random probability measures on~$\PP(\XXX)$ and let~$\zeta_k^m$ be the image of~$ \zzeta_k$ under the projection~$p_m$. The following theorem is a particular case of a more general result established in~\cite{DG-1987} (see also Theorem~4.6.1 in~\cite{ADOZ00}). 

\begin{theorem} \label{t5.2} 
Suppose that for any integer~$m\ge1$ the sequence~$\{\zeta_k^m\}$ satisfies the LDP with a good rate function~$I_m:\PP(X^m)\to[0,+\infty]$. Then the LDP holds for~$\{\zzeta_k\}$ with the good rate function 
\begin{equation} \label{5.4}
\III(\ssigma)=\sup_{m\ge1}I_m\bigl(\ssigma\circ p_m^{-1}\bigr). 
\end{equation}
\end{theorem}

\begin{proof}
 \medskip
{\it Step~1: Rate function}.
Let us prove that the function~$\III$ defined by~\eqref{5.4} is a good rate function. Indeed, since~$I_m$ are good rate functions, for any $\alpha\in\R$ we have
\begin{align} 
\{\III\le\alpha\}
=\bigcap_{m=1}^\infty\bigl\{\ssigma\in\PP(\XXX):
I_m\bigl(\ssigma\circ p_m^{-1}\bigr)\le\alpha\bigr\}
=\bigcap_{m=1}^\infty\{\ssigma\circ p_m^{-1}\in K_\alpha^m\},
\label{5.5}
\end{align}
where $K_\alpha^m$ are compact subsets in~$\PP(X^m)$. This relation immediately implies that the set $\{\III\le\alpha\}$ is closed and therefore~$\III$ is lower semicontinuous. Furthermore, since a sequence $\{\ssigma_j\}\subset\PP(\XXX)$ converges if and only if so does $\{\ssigma_j\circ p_m^{-1}\}$ for any~$m\ge1$, it follows from~\eqref{5.5} that the level sets of~$\III$ are compact. 

\smallskip
{\it Step~2: Lower bound}.
Let $G\subset\PP(\XXX)$ be an open subset. It suffices to prove that, for any $\ssigma\in G$, we have
\begin{equation*}
\liminf_{k\to\infty}\frac1k\log\IP\{\zzeta_k\in G\}\ge-\III(\ssigma). 
\end{equation*}
Since $G$ is open, for any $\ssigma\in G$, one can find an integer $m\ge1$ and open subset $G_m\subset\PP(X^m)$ containing~$\ssigma\circ p_m^{-1}$ such that $G\supset p_m^{-1}(G_m)$. Since the LDP holds for $\zeta_k^m=\zzeta_k\circ p_m^{-1}$, it follows that 
\begin{align*}
\liminf_{k\to\infty}\frac1k\log\IP\{\zzeta_k\in G\}
&\ge\liminf_{k\to\infty}\frac1k\log\IP\{\zzeta_k\in p_m^{-1}(G_m)\}\\
&=\liminf_{k\to\infty}\frac1k\log\IP\{\zeta_k^m\in G_m\}\ge-I_m(G_m).
\end{align*}
It remains to note that $I_m(G_m)\le I_m(\ssigma\circ p_m^{-1})\le \III(\ssigma)$. 

\smallskip
{\it Step~3: Upper bound}.
Let $F\subset\PP(\XXX)$ be a closed subset. It suffices to prove that, if $\alpha<\III(F)$, then 
\begin{equation} \label{5.6}
\liminf_{k\to\infty}\frac1k\log\IP\{\zzeta_k\in F\}\le-\alpha. 
\end{equation}
Relation~\eqref{5.5} implies that 
$$
\varnothing=F\cap\{\III\le\alpha\}=\bigcap_{m=1}^\infty
F\cap \bigl\{I_m\bigl(\ssigma\circ p_m^{-1}\bigr)\le\alpha\bigr\}.
$$
Since $F\cap\{\III\le\alpha\}$ is a compact set, it follows that one can find an integer $m\ge1$ such that $F\cap \bigl\{I_m\bigl(\ssigma\circ p_m^{-1}\bigr)\le\alpha\bigr\}=\varnothing$. Denoting by~$F_m$ the image of~$F$ under the projection~$p_m$, we conclude that 
$I_m(F_m)>\alpha$. Since $F\subset p_m^{-1}(F_m)$, using the LDP for~$\zeta_k^m$, we derive
\begin{align*}
\liminf_{k\to\infty}\frac1k\log\IP\{\zzeta_k\in F\}
&\le \liminf_{k\to\infty}\frac1k\log\IP\{\zzeta_k\in p_m^{-1}(F_m)\}\\
&=\liminf_{k\to\infty}\frac1k\log\IP\{\zeta_k^m\in F_m\}
\le -I_m(F_m)<-\alpha. 
\end{align*}
This completes the proof of~\eqref{5.6} and of the theorem.
\end{proof}

Theorem~\ref{t5.2} admits a simple generalisation to the case of uniform LDP. Namely, let us assume that we are given a sequence of random probability measures~$\{\zzeta_k(y)\}$ on~$\XXX$ depending on a parameter $y\in Y$, where~$Y$ is an arbitrary set. We shall say that~$\{\zzeta_k(y)\}$ satisfies the {\it uniform LDP\/} with a good rate function $\III:\PP(\XXX)\to[0,+\infty]$ if 
\begin{align} 
-\III(\dot\Gamma)
&\le\liminf_{k\to\infty}\frac1k\log\inf_{y\in Y}\IP\{\zzeta_k(y)\in\Gamma\}\notag\\
&\le\liminf_{k\to\infty}\frac1k\log\sup_{y\in Y}\IP\{\zzeta_k(y)\in\Gamma\}
\le-\III(\overline\Gamma),\label{6.8}
\end{align}
where $\Gamma\subset\PP(\XXX)$ is an arbitrary Borel subset. The proof of the following result literally repeats that of Theorem~\ref{t5.2}, and we omit it. 

\begin{theorem} \label{t6.4} 
Suppose that for any integer~$m\ge1$ the sequence~$\{\zeta_k^m(y),y\in Y\}$ satisfies the uniform LDP with a good rate function~$I_m:\PP(X^m)\to[0,+\infty]$. Then the uniform LDP holds for~$\{\zzeta_k(y),y\in Y\}$ with the good rate function~\eqref{5.4}.
\end{theorem}

\addcontentsline{toc}{section}{Bibliography}
\def\cprime{$'$} \def\cprime{$'$}
  \def\polhk#1{\setbox0=\hbox{#1}{\ooalign{\hidewidth
  \lower1.5ex\hbox{`}\hidewidth\crcr\unhbox0}}}
  \def\polhk#1{\setbox0=\hbox{#1}{\ooalign{\hidewidth
  \lower1.5ex\hbox{`}\hidewidth\crcr\unhbox0}}}
  \def\polhk#1{\setbox0=\hbox{#1}{\ooalign{\hidewidth
  \lower1.5ex\hbox{`}\hidewidth\crcr\unhbox0}}} \def\cprime{$'$}
  \def\polhk#1{\setbox0=\hbox{#1}{\ooalign{\hidewidth
  \lower1.5ex\hbox{`}\hidewidth\crcr\unhbox0}}} \def\cprime{$'$}
  \def\cprime{$'$} \def\cprime{$'$} \def\cprime{$'$}
\providecommand{\bysame}{\leavevmode\hbox to3em{\hrulefill}\thinspace}
\providecommand{\MR}{\relax\ifhmode\unskip\space\fi MR }
\providecommand{\MRhref}[2]{%
  \href{http://www.ams.org/mathscinet-getitem?mr=#1}{#2}
}
\providecommand{\href}[2]{#2}

\end{document}